\documentclass[a4paper,10pt]{article}
\usepackage[utf8]{inputenc}
\usepackage{amsmath,amssymb,amsthm}%,amsart}
\usepackage{graphicx,graphics,color,enumitem,booktabs}
\usepackage{epstopdf}
\usepackage{fancyhdr}
\usepackage{titlesec}
\usepackage{url}
\usepackage{bm}
\usepackage{tikz,ifpdf}
\usepackage{xifthen}
\usepackage{xparse}
\newtheorem{proposition}{Proposition}[section]

\newtheorem{theorem}{Theorem}[section]
\newtheorem{lemma}{Lemma}[section]
\newtheorem{remark}{Remark}[section]
\newtheorem{definition}{Definition}[section]
\usepackage{marginnote}

\newcommand{\aast}{a^{\ast}}
\newcommand{\Btensor}{\mathcal{B}}

\newcommand{\dell}{\partial}
\newcommand{\dellmat}{\dell^{\bullet}}
\DeclareMathOperator{\dist}{dist}
\DeclareMathOperator{\supp}{supp}

\newcommand{\diff}{\frac{\d}{\d t}}
\newcommand{\Ga}{\Gamma}

\newcommand{\Gat}{\Gamma(t)}
\newcommand{\GT}{\mathcal{G}_T}

\newcommand\laplace{\Delta}

\newcommand{\nbg}{\nabla_{\Gamma}}
\newcommand{\nbgh}{\nabla_{\Gamma_h}}
\newcommand{\mat}{\partial^{\bullet}}

\newcommand{\co}{continuous}
\def \d {\mathrm{d}}
\newcommand{\disp}{\displaystyle}
\newcommand{\eps}{\varepsilon}

\newcommand{\N}{\mathbb{N}}
\newcommand{\nb}{\nabla}

\newcommand{\pa}{\partial}
\newcommand{\R}{\mathbb{R}}

\newcommand{\spn}{\textnormal{span}}
\newcommand{\st}{such that}

\def \t {(t)}
\newcommand{\Th}{\mathcal{T}_h}
\def \to {\rightarrow}
\newcommand{\vphi}{\varphi}
\DeclareMathOperator{\linspan}{lin}

\newcommand\Linfty[2][]{%
  \lVert #2\rVert_{%
    L^{\infty}%
    \ifx\\#1\\\else(#1)\fi%
  }%
}
\newcommand{\Linftyy}[2][]{%
  \lVert #2\rVert_{
    L^{\infty}%
    \ifthenelse{\isempty{#1}}{}{(#1)}
  }
}
\newcommand\Ltwo[2][]{%
  \lVert #2\rVert_{%
    L^{2}%
    \ifx\\#1\\\else(#1)\fi%
  }%
}
\newcommand\Lone[2][]{%
  \lVert #2\rVert_{%
    L^{1}%
    \ifx\\#1\\\else(#1)\fi%
  }%
}
\newcommand\abs[1]{ \lvert #1\rvert }
\newcommand\bigabs[1]{ \bigl\lvert #1\bigr\rvert }

\DeclareDocumentCommand{\wtL}{O{\alpha} O{} m}{
  \lVert #3\rVert_{L^{2},\ifthenelse{\isempty{#2}}{}{#2,}#1}
}
\DeclareDocumentCommand{\wtH}{O{\alpha} O{} m}{
  \lVert #3\rVert_{H^{1},\ifthenelse{\isempty{#2}}{}{#2,}#1}
}
\DeclareDocumentCommand{\BigwtH}{O{\alpha} O{} m}{
  \Bigl\lVert #3\Bigr\rVert_{H^{1},\ifthenelse{\isempty{#2}}{}{#2,}#1}
}
\DeclareDocumentCommand{\lrwtH}{O{\alpha} O{} m}{
  \left\lVert #3\right\rVert_{H^{1},\ifthenelse{\isempty{#2}}{}{#2,}#1}
}
\DeclareDocumentCommand{\wtHk}{O{k} O{\alpha} m}{
  \lVert #3\rVert_{H^{#1},#2}
}
\DeclareDocumentCommand{\Lnorm}{%
  O{} O{2} m%
}{\lVert #3\rVert_{%
    L^{#2}%
    \ifx\\#1\\\else(#1)\fi%
  }%
}
\DeclareDocumentCommand{\Wnorm}{%
  O{} O{1,\infty} m%
}{\lVert #3\rVert_{%
    W^{#2}%
    \ifx\\#1\\\else(#1)\fi%
  }%
}

\DeclareDocumentCommand{\Hnorm}{%
  O{} O{1} m%
}{\lVert #3\rVert_{%
    H^{#2}%
    \ifx\\#1\\\else(#1)\fi%
  }%
}

\newcommand\ddxi[2]{\frac{\partial #1}{\partial #2}}
\newcommand\ddti[2]{\frac{\mathrm{d}#1}{\mathrm{d}#2}}
\newcommand\green{\mathbf{G}}

\newcommand\surface{\Ga}

\newcommand\lb{\laplace_{\Ga}}

\newcommand\Lproj{P_{0}}

\newcommand\Rmap{P_{h,1}}
\newcommand\lRmap{P_{1}}
\newcommand\lrmap{\lRmap}
\newcommand\Rmapl{\lRmap}
\newcommand\ipol{\widetilde{I_{h}}}
\newcommand\ipoll{I_{h}}
\newcommand\pwt{\sigma}

\newcommand\riemdist{\textnormal{dist}_{\surface\t}}

\makeatletter
\newcommand{\oset}[3][0ex]{%
  \mathrel{\mathop{#3}\limits^{
    \vbox to#1{\kern-2\ex@
    \hbox{$\scriptstyle#2$}\vss}}}}
\makeatother

\DeclareDocumentCommand{\HspaceO}{%
  O{1} m%
}{\oset[0ex]{\circ\hphantom{#1}}{H^{#1}}\!\! (#2)%
}

\DeclareDocumentCommand{\WspaceO}{%
  O{k,p} m%
}{\oset[0ex]{\circ\hphantom{#1}}{W^{#1}}\!\! (#2)%
}

\DeclareDocumentCommand{\Hspace}{%
  O{1} m%
}{H^{#1}(#2)%
}
\DeclareDocumentCommand{\Wspace}{%
  O{k,p} m%
}{W^{#1}(#2)%
}

\begin{document}

\title{Maximum norm stability and error estimates for the evolving surface
  finite element method}
\author{Bal\'{a}zs Kov\'{a}cs\footnote{Mathematisches Institut, University of
  T\"{u}bingen, Auf der Morgenstelle 10, 72076 T\"{u}bingen, Germany, E-mail
  address: $\{$kovacs,power$\}$@na.uni-tuebingen.de} \, and
Christian Andreas Power Guerra\footnotemark[\value{footnote}]
}
\maketitle

\begin{abstract}
  We show convergence in the natural $L^{\infty}$- and $W^{1,\infty}$-norm for a
  semidiscretization with linear finite elements of a linear parabolic partial
  differential equations on
  evolving surfaces.  To prove this we show error estimates for a Ritz map,
  error estimates for the material derivative of a Ritz map and a weak discrete
  maximum principle.
\end{abstract}


\section{Introduction}
% \label{sec:introduction}

Many important problems can be modeled by partial differential equations \linebreak (PDEs)
on evolving surfaces.  Examples for such equations are given in
material sciences, fluid mechanics and biophysics
\cite{Fife,2004james_lowengrub,2010elliott_stinner}.  The basic linear parabolic
PDE on a moving surface is
\begin{equation*}
  \mat u + u \nb_{\Gat} \cdot v - \laplace_{\Gat} u = f \quad
  \text{on } \surface(t).
\end{equation*}
Here the velocity \(v\) is explicitly given and we seek to compute a numerical approximation to the exact
solution \(u\).  Dziuk and Elliott \cite{DziukElliott_SFEM} introduced the evolving
surface finite element method (ESFEM) to solve this problem.  Error estimates
for the semidiscretization with piecewise linear finite elements in
the \(L^{2}\)- and \(H^{1}\)-norm are given in
\cite{DziukElliott_L2,DziukElliott_ESFEM}.

\bigskip
The aim of this work is to give error bounds for the semidiscretization with linear finite elements in the
\(L^{\infty}\)- and \(W^{1,\infty}\)-norm. The authors are not aware of any other maximum norm convergence results for \emph{evolving surface} PDEs.

Such estimates are of interest for nonlinear
parabolic PDEs on evolving surfaces and if the velocity
\(v\) is not explicitly given, but depends on the exact solution \(u\).
Example of such problems are given in
\cite{Deckelnick2001,Fife,BarrettGarckeNurnberg,ElliottStyles_ALEnumerics,CGG}
and the references therein. The first convergence results for such coupled problems have been recently shown in \cite{soldriven}. The treatment of such general equations are beyond the scope of this paper.

\bigskip
Our convergence proof for the semidiscretization of the linear heat equation on evolving surfaces relies on three main results.  

\begin{itemize}
\item 
We give some error bounds in the \(L^{\infty}\)- and \(W^{1,\infty}\)-norms for a suitable \emph{time dependent} Ritz map (also used in \cite{LubichMansour_wave}, which is not the same as the one in \cite{DziukElliott_L2,Demlow2009}). The proofs of these results are based on Nitsche's weighted norm technique \cite{Nitsche}.

\item 
Since the surface evolves in time the Ritz map is time dependent, hence it does not commute with the time derivative. We therefore need the essential novel results: the \(L^{\infty}\)- and \(W^{1,\infty}\)-norm error bounds in the \emph{material derivatives} of the Ritz map. Up to our knowledge such maximum norm estimates have not been shown in the literature until now.

\item 
We extend the weak finite element maximum principle, which is originally due to Schatz, Thom\'{e}e and Wahlbin \cite{1980_stw} for Euclidian domains, to the evolving surface case.  In \cite{1980_stw} they use basic properties of the semigroup corresponding to the linear heat equation on a bounded domain.  Since there is no semigroup theory for the linear heat equation on evolving surfaces we are going to use a different approach.
\end{itemize}

\medskip
%\rev{Theorem 4.1 ($L^\infty$ and $W^{1,\infty}$ estimates for static surface FEM) is both suboptimal and already essentially proved with correct logarithmic factors elsewhere. In parallel with the Euclidean case there should be no logarithmic factors in the $W^{1,\infty}$ estimate and one in the $L^\infty$ case. The paper [Rannacher \& Scott (1982)] contains the first proof of logarithm-free $W^{1,\infty}$ bounds and refers to several papers which have the correct single logarithmic factor in the $L^\infty$ bound; cf.\ [Haverkamp (1984)] or [Schatz (1998)].}{
We expect that the results presented here may be improved to have optimal logarithmic factors, shown using more involved proof techniques generalised from the Euclidean domain case, see for instance \cite{Haverkamp,RannacherScott,Schatz} and especially the proof of the logarithm-free discrete maximum principle proved in \cite{1998_stw}. However, such logarithmically optimal bounds are not in the scope of the present work, since such a refined analysis would easily double the length of the paper.
%}

%\rev{A recent preprint of H.~Kr\"{o}ner (available on arXiv) also contains $L^\infty$ bounds for surface FEM using a different proof technique.}{
In a recent preprint of Kr\"{o}ner \cite{Kroner}, $L^\infty$ estimates -- of order $O(|\log(h)|h + \tau^{1/2})$ -- are shown for full discretisations of parabolic PDEs on stationary surfaces. The results of that paper are obtained by using different proof techniques.%, independent of \cite{1980_stw,1998_stw}.
%}

\medskip
The layout of the paper is as follows.  We begin in
Section~\ref{sec:parab-probl-evolv} by fixing some notation and introducing the
most basic notion.
In the first three subsection of Section~\ref{sec:preliminaries} we quickly
develop the evolving surface finite element method (ESFEM) and recall basic
results and estimates.  In the following three subsection we introduce a surface
version of Nitsche's weighted norms and finish with an \(L^{2}\)-projection.
In Section~\ref{sec:ritz-map-some} we give error bound in the maximum norm for
our Ritz map.  In Section~\ref{sec:weak-discr-maxim-1} we derive a weak
ESFEM maximum principle.  In Section~\ref{sec:conv-semid} we give error
bounds for the semi discretization of the linear heat equation on evolving
surfaces in the \(L^{\infty}\)- and \(W^{1,\infty}\)-norm.  In
Section~\ref{sec:numerical-experiment} we present the results of a numerical
experiment.  We gather technical details for calculations with our weight functions in
Appendix~\ref{sec:extens-some-auxil}.
% ; consider in particular \eqref{eq:3} and \eqref{eq:41}.

\section{A parabolic problem on evolving surfaces}
\label{sec:parab-probl-evolv}

Let us consider a smooth evolving closed hypersurface $\Ga\t \subset \R^{m+1}$
(our main focus is on the case $m=2$, but some of our results hold for more general cases), $0 \leq t \leq T$, which moves with a given smooth velocity
$v$. More precise we assume that there exists a smooth dynamical system
\(\Phi\colon \surface_{0}\times [0,T]\to \R^{m+1}\), such that for each \(t\in
[0,T]\) the map \(\Phi_{t}:= \Phi(\, .\, , t)\) is an embedding.  We define
\(\surface(t):= \Phi_{t}(\surface_{0})\) and define the velocity \(v\) via the
equation \(\dell_{t}\Phi(x,t)= v\bigl(\Phi(x,t),t\bigr)\). Let $\mat u = \pa_{t}
u + v \cdot \nb u$ denote the material derivative of the function $u$. The
tangential gradient is given by $\nbg u = \nb u -\nb u \cdot \nu \nu$, where
\(\nu\) is the unit normal and finally we define the Laplace--Beltrami operator
via \(\lb u = \nbg \cdot \nbg u\).  This article shares the setting
of Dziuk and Elliott \cite{DziukElliott_ESFEM,DziukElliott_L2}, and \cite{diss_Mansour}.

We consider the following linear problem derived in \cite[Section~3]{DziukElliott_ESFEM}:
\begin{equation}
    \label{eq_ES-PDE-strong-form}
    \begin{cases}
        \begin{alignedat}{3}
          \mat u + u \nb_{\Gat} \cdot v - \laplace_{\Gat} u =&\ f & \qquad & \textrm{ on } \Ga\t ,\\
            u(.,0) =&\ u_0 & \qquad & \textrm{ on } \Ga(0).
        \end{alignedat}
    \end{cases}
\end{equation}
We use Sobolev spaces on surfaces: For a sufficiently smooth surface
$\Gamma$ and $1\leq p \leq \infty$ we define
\begin{equation*}
  W^{1,p}(\Gamma) = \bigl\{ \eta \in L^{p}(\Gamma) \mid \nbg \eta \in
  L^{p}(\Gamma)^{m+1} \bigr\},
\end{equation*}
and analogously $W^{k,p}(\Gamma)$ for $k\in \N$
\cite[Section~2.1]{DziukElliott_ESFEM}. We set \(H^{k}(\Gamma) =
W^{k,2}(\Gamma)\).  Finally, $\GT$ denotes the space-time
manifold, i.e.\ $\GT:= \cup_{t\in[0,T]}
\Ga\t\times\{t\}$. \par
If \(f=0\) then a weak formulation of this problem reads as follows.
\begin{definition}[weak solution, \cite{DziukElliott_ESFEM} Definition 4.1]
% \label{def_ES-PDE-weak}
A function $u\in H^1(\GT)$ is called a \emph{weak solution} of
\eqref{eq_ES-PDE-strong-form}, if for almost every $t\in[0,T]$
\begin{equation*}
%\label{eq_ES-PDE-weak-form}
    \diff \int_{\Gat}\!\!\!\! u \vphi + \int_{\Gat}\!\!\!\! \nb_{\Gat} u
    \cdot \nb_{\Gat} \vphi = \int_{\Gat}\!\!\!\! u\mat \vphi
\end{equation*}
holds for every  $\vphi \in H^1(\GT)$ and $u(.,0)=u_0$.
\end{definition}

For suitable $f$ and $u_{0}$ existence and uniqueness results, for the strong
and the weak problem, were obtained in \cite[Section~4]{DziukElliott_ESFEM}.  \par
Throughout this article we assume that \(f\) and \(u_{0}\) a such regular that \(u\in
W^{3,\infty}(\GT)\).  Furthermore we set for simplicity reasons in all sections
$f=0$, since the extension of our results to the inhomogeneous case are
straightforward.

\section{Preliminaries}
\label{sec:preliminaries}

We give a summary of this section.  In Section~\ref{sec:semi-dis-creti} we
introduce the ESFEM, which is due to Dziuk and Elliott
\cite{DziukElliott_ESFEM}.  In Section~\ref{subsection: lift} we recall the
lifting process, which originates in Dziuk \cite{Dziuk88}.  In
Section~\ref{sec:geom-estim-bilin} we collect important results from Dziuk
and Elliott \cite{DziukElliott_L2} and sometimes state them in a slightly more
general fashion.  In Section~\ref{sec:weight-norms-estim} we introduce
weighted norms, which are due to Nitsche \cite{Nitsche}, and give connections to
the \(L^{\infty}\)-norm.  In
Section~\ref{sec:interp-inverse-estim} we give interpolation estimates in the
\(L^{2}\)-, \(L^{\infty}\)- and weighted norms and further give some special
interpolation estimates in weighted norms.  The latter two were first stated in
Nitsche \cite{Nitsche}.  In
Section~\ref{sec:estimates-an-l2} we introduce an \(L^{2}\)-projection, give
a stability bound in \(L^{p}\)-norms and finish with a error estimate with respect to a
different weight function.  The basic reference for this is Douglas, Dupont,
Wahlbin \cite{1975_ddw} and Schatz, Thom\'{e}e, Wahlbin \cite{1980_stw}.

\subsection{Semidiscretization with the evolving surface finite
  element method}
\label{sec:semi-dis-creti}

The smooth surface $\Gat$ is approximated by a triangulated one denoted by $\Ga_h(t)$, whose vertices $a_j\t = \Phi(a_{j}(0),t)$ are sitting on the surface for all time, such that
\begin{equation*}
    \Ga_h(t) = \bigcup_{E(t)\in \Th(t)} E(t).
\end{equation*}
We always assume that the (evolving) simplices $E(t)$ are forming an admissible
triangulation $\Th(t)$, with $h$ denoting the maximum diameter.  Admissible triangulations were introduced in \cite[Section~5.1]{DziukElliott_ESFEM}: Every $E(t)\in \Th(t)$ satisfies that the inner radius $\sigma_{h}$ is bounded from below by $ch$ with $c>0$, and $\Ga_{h}(t)$ is not a global double covering of $\Ga(t)$. The discrete tangential gradient on the discrete surface $\Ga_h\t$ is given by
\begin{equation*}
    \nb_{\Ga_h\t} \phi := \nb {\phi} - \nb {\phi} \cdot \nu_h \nu_h,
\end{equation*}
understood in a piecewise sense, with $\nu_h$ denoting the normal to $\Ga_h(t)$ (see \cite{DziukElliott_ESFEM}).

For every $t\in[0,T]$ we define the finite element subspace $S_h\t$ spanned by the continuous, piecewise linear evolving basis functions $\chi_j$, satisfying
\[
\chi_j(a_i\t,t) = \delta_{ij} \quad  \text{for all }i,j = 1, 2, \dotsc, N,
\]
therefore
%\begin{equation*}
$S_h\t = \spn\big\{ \chi_1( \, . \,,t), \chi_2( \, . \,,t), \dotsc, \chi_N( \, . \,,t) \big\}$.
%\end{equation*}

We interpolate the dynamical system \(\Phi\) by
\(\Phi_{h}\colon \surface_{h}(0) \to \R^{m+1}\), the discrete dynamical system of \(\surface_{h}(t)\).
This defines a discrete surface velocity \(V_{h}\) via \(\dell_{t} \Phi_{h}(y_{h},t) = V_{h}\bigl( \Phi_{h}(y_{h},t),t\bigr)\). Then the discrete material derivative is given by
\begin{equation*}
    \mat_h \phi_h = \pa_t \phi_h + V_h \cdot \nb \phi_h  \qquad (\phi_h \in S_h\t).
\end{equation*}

The key \textit{transport property} derived in \cite[Proposition 5.4]{DziukElliott_ESFEM}, is the following
\begin{equation}\label{eq: transport property}
    \mat_h \chi_k = 0 \qquad \textrm{for} \quad k=1,2,\dotsc,N.
\end{equation}

The spatially discrete problem for evolving surfaces is: Find a $U_h\in S_h\t$ with $\mat_h U_h\in S_h\t$ and temporally smooth \st, for every $\phi_h \in S_h\t$  with $\mat_h \phi_h \in S_h\t$,
\begin{equation}
\label{eq: semidiscrete problem}
        \diff \!\!\int_{\Ga_h\t}\!\! U_h \phi_h
        + \int_{\Ga_h\t}\!\!\! \nbgh U_h \cdot \nbgh \phi_h = \int_{\Ga_h\t}\!\! U_h \mat_h \phi_h ,
\end{equation}
with the initial condition $U_h( \, . \,,0)=U_h^0\in S_h(0)$ being a sufficient approximation to $u_0$.

\subsection{Lifts}
\label{subsection: lift}

In the following we recall the so called \emph{lift operator}, which was introduced in \cite{Dziuk88} and further investigated in \cite{DziukElliott_ESFEM,DziukElliott_L2}. The lift operator projects a finite element function on the discrete surface onto a
function on the smooth surface.

Using the \emph{oriented distance function} $d$ (\cite[Section~2.1]{DziukElliott_ESFEM}), for a \co\ function $\eta_h \colon \Ga_h\t \to \R$ its lift is define as
\begin{equation*}
    \eta_{h}^{l}(x^{l},t) := \eta_h(x,t), \qquad x\in\Ga\t,
\end{equation*}
where for every $x\in \Ga_{h}\t$ the value $x^{l}=x^{l}(x,t)\in\Gat$ is uniquely defined via $x = x^{l} + \nu(x^{l},t) d(x,t)$. This notation for $x^l$ will also be used later on. By $\eta^{-l}$ we mean the function whose lift is $\eta$, and by \(E_{h}^{l}\) we mean the lift of the triangle \(E_{h}\).

The following pointwise estimate was shown in the proof of Lemma~3 from Dziuk
\cite{Dziuk88}:
\begin{equation}
  \label{eq:66}
  \frac{1}{c} \abs{\nbg \eta_{h}^{l}(x^{l})}
  \leq \abs{\nbgh \eta_{h}(x)}
  \leq c \abs{\nbg \eta_{h}^{l}(x^{l})}.
\end{equation}

We now recall some notions using the lifting process from \cite{Dziuk88,DziukElliott_ESFEM}. We have the lifted finite element space
\begin{equation*}
    S_h^l\t := \big\{ \vphi_h = \phi_h^l \, | \, \phi_h\in S_h\t \big\}.
\end{equation*}
By $\delta_h$ we denote the quotient between the \co\ and discrete surface measures, $\d A$ and $\d A_h$, defined as $\delta_h \d A_h = \d A$.
% Further, we recall that
% \begin{equation*}
%     \pr := \big(\delta_{ij} - \nu_{i}\nu_{j}\big)_{i,j=1}^{m+1} \quad \textrm{and} \quad \prh := \big(\delta_{ij} - \nu_{h,i}\nu_{h,j}\big)_{i,j=1}^{m+1}
% \end{equation*}
% are the projections onto the tangent spaces of $\Ga$ and $\Ga_h$. Further, from \cite{DziukElliott_L2}, we recall the notation
% \begin{equation*}
%     Q_h = \frac{1}{\delta_h} (I-d\wein) \pr \prh \pr (I-d\wein),
% \end{equation*}
% where $\wein$ ($\wein_{ij} = \pa_{x_j}\nu_i$) is the (extended) Weingarten map.
For these quantities we recall some results from \cite[Lemma 5.1]{DziukElliott_ESFEM}, \cite[Lemma 5.4]{DziukElliott_L2} and \cite[Lemma 6.1]{diss_Mansour}.
\begin{lemma}
%\label{lemma: geometric est}
  % Assume that $\Ga_h\t$ and $\Ga\t$ is from the above setting, then we have the estimates:
  For sufficiently small $h$, we have the estimates
  \begin{gather*}
    \|d\|_{L^\infty(\Ga_h\t)} \leq c h^2,
    \quad \|1-\delta_h\|_{L^\infty(\Ga_h\t)} \leq c h^2,
    % \\
    % \|\mat_h d \|_{L^\infty(\Ga_h\t)} \leq c h, \quad \|\pr - Q_h\|_{L^\infty(\Ga_h\t)} \leq c h^2, \quad \|\pr(\mat_hQ_h)\pr\|_{L^\infty(\Ga_h\t)} \leq c h^2 ,
  \end{gather*}
  with constants independent of \(t\) and \(h\).
  % depending on $\GT$, but not on $t$.
\end{lemma}

\subsection{Geometric estimates and bilinear forms}
\label{sec:geom-estim-bilin}

Let us denote by \(\Phi_{h}^{l}\colon \surface_{0}\times [0,T]\to \R^{m+1}\) the
lift of \(\Phi_{h}\). We define the velocity $v_{h}$ via the formula
\(\dell_{t}\Phi_{h}^{l}(x,t) = v_{h}\bigl(\Phi_{h}^{l}(x,t),t\bigr)\).  Then the
discrete material derivative on \(\surface(t)\) is given by
\begin{equation*}
  \mat_{h} u = \dell_{t} u + v_{h} \cdot \nb u,
\end{equation*}
which satisfies the following relations, cf.~\cite{DziukElliott_L2}:
\begin{gather}
  \label{eq:40}
  \mat u
  = \mat_{h}u + (v_{h}-v)\cdot \nbg u, \\
  \label{eq: velocity L infty bound}
  \Lnorm[\surface(t)][\infty]{v - v_{h}}
  + h \Wnorm[\surface(t)][\infty]{v-v_{h}}
  \leq c h^{2} \Wnorm[\surface(t)][2,\infty]{v},
\end{gather}

We use the time dependent bilinear forms defined in \cite[Section 3.3]{DziukElliott_L2}: for $z,\vphi \in H^1(\Ga\t)$ and $Z_h, \phi_h \in H^{1}(\surface_{h}(t))$:
\begin{align*}
\begin{aligned}[c]
    a(t;z,\vphi)   &= \int_{\Ga\t}\!\!\! \nbg z \cdot \nbg \vphi, \\
    m(t;z,\vphi)   &= \int_{\Ga\t}\!\!\! z \vphi, \\
    g(t;v;z,\vphi) &= \int_{\Ga\t}\!\!\! (\nbg \cdot v) z\vphi,\\
    % b(t;v;z,\vphi) &= \int_{\Ga\t}\!\!\! \Btensor(v) \nbg z \cdot \nbg \vphi,
\end{aligned}
& \quad
\begin{aligned}[c]
    a_h(t;Z_h,\phi_h)     &= \sum_{E\in \Th}\! \int_E \!\!\! \nbgh Z_h \cdot \nbgh \phi_h, \\
    m_h(t;Z_h,\phi_h)     &= \int_{\Ga_h\t}\!\!\! Z_h \phi_h\\
    g_h(t;V_h;Z_h,\phi_h) &= \int_{\Ga_h\t}\!\!\! (\nbgh \cdot V_h) Z_h \phi_h,\\
    % b_h(t;V_h;Z_h,\phi_h) &= \sum_{E\in \Th}\! \int_E \!\!\! \Btensor_h(V_h) \nbgh Z_h \cdot \nbgh \phi_h,
\end{aligned} \\
%\end{align*}
%\begin{align*}
%   % CAPG: To avoid the warning message. Let's thing about this later, when we
%   % know which journal we want.
%   a(t;z,\vphi)   &= \int_{\Ga\t} \nbg z \cdot \nbg \vphi, \\
%   m(t;z,\vphi)   &= \int_{\Ga\t} z \vphi, \\
%   g(t;v;z,\vphi) &= \int_{\Ga\t} (\nbg \cdot v) z\vphi,\\
  b(t;v;z,\vphi) &= \int_{\Ga\t} \Btensor(v) \nbg z \cdot \nbg \vphi,\\
%   a_h(t;Z_h,\phi_h)     &= \sum_{E\in \Th} \int_E \nbgh Z_h
%                           \cdot \nbgh \phi_h, \\
%   m_h(t;Z_h,\phi_h)     &= \int_{\Ga_h\t} Z_h \phi_h\\
%   g_h(t;V_h;Z_h,\phi_h) &= \int_{\Ga_h\t} (\nbgh \cdot V_h) Z_h \phi_h,\\
  b_h(t;V_h;Z_h,\phi_h) &= \sum_{E\in \Th} \int_E \Btensor_h(V_h) \nbgh Z_h
%                           \cdot \nbgh \phi_h,
\end{align*}
where the discrete tangential gradients are understood in a piecewise sense, and
with the matrices
% CAPG: Bad box message.
% \begin{alignat*}{3}
%     \Btensor(v)_{ij} &= \delta_{ij} (\nbg \cdot v) - \big( (\nbg)_i v_j + (\nbg)_j v_i \big), \quad && (i,j=1,2,\dotsc,m+1),\\
%     \Btensor_h(V_h)_{ij} &= \delta_{ij} (\nbg \cdot V_h) - \big( (\nbgh)_i (V_h)_j + (\nbgh)_j (V_h)_i \big), \quad && (i,j=1,2,\dotsc,m+1).
%   \end{alignat*}
\begin{gather*}
  \Btensor(v)_{ij} = \delta_{ij} (\nbg \cdot v) -
                     \big( (\nbg)_i v_j + (\nbg)_j v_i \big), \\
  \Btensor_h(V_h)_{ij} = \delta_{ij} (\nbg \cdot V_h)
                         - \big( (\nbgh)_i (V_h)_j + (\nbgh)_j (V_h)_i \big),
\end{gather*}
where $i,j= 1,2,\dotsc, m+1$.

The time derivatives of the bilinear forms are given in the following lemma.
\begin{lemma}[Discrete transport property]
\label{lemma: transport properties}
  For \(z,\varphi\in H^{1}(\surface(t))\) we have
  \begin{equation}
    \label{eq: continuous transport properties}
    \begin{aligned}
      \frac{\d}{\d t} m(z,\vphi)
      & = m(\mat_{h} z, \vphi) + m(z, \mat_{h} \vphi)
      + g(v_{h}; z, \vphi),\\
      \frac{\d}{\d t} a(z,\vphi)
      & = a(\mat_{h} z, \vphi) + a(z, \mat_{h} \vphi)
      + b(v_{h}; z, \vphi).
    \end{aligned}
  \end{equation}
  Similarly for $Z_{h},\phi_{h}\in H^{1}(\surface_{h}(t))$ we have
  \begin{equation}
    \label{eq: discrete transport properties}
    \begin{aligned}
      \ddti{}{t} m_{h}(Z_{h},\phi_{h})
      &= m_{h}(\dellmat Z_{h},\phi_{h}) + m_{h}( Z_{h},\dellmat \phi_{h})
      + g_{h}(V_{h};Z_{h},\phi_{h}), \\
      \ddti{}{t} a_{h}(Z_{h},\phi_{h})
      &= a_{h}(\dellmat Z_{h},\phi_{h}) + a_{h}(Z_{h},\dellmat \phi_{h})
      + b_{h}(V_{h}; Z_{h},\phi_{h}).
    \end{aligned}
  \end{equation}
\end{lemma}

Important and often used results are the bounds of the geometric perturbation errors in the bilinear forms.
\begin{lemma}
\label{lemma: geometric errors of forms - Holder}
    For all $1\leq p,q\leq \infty$, that are conjugate, $p^{-1}+q^{-1}=1$, and for arbitrary $Z_{h}\in  L^{p}\bigl(\surface_{h}(t)\bigr)$ and $\phi_{h}\in L^{q}\bigl(\surface_{h}(t)\bigr)$, with corresponding lifts $z_{h}\in  L^{p}\bigl(\Ga(t)\bigr)$ and $\vphi_{h}\in L^{q}\bigl(\Ga(t)\bigr)$ we have the following estimates:
    \begin{align*}
        \abs{m(z_h,\vphi_h) - m_{h}(Z_{h},\phi_{h})} & \leq c h^{2} \Lnorm[\surface(t)][p]{z_h} \Lnorm[\surface(t)][q]{\vphi_h}, \\
        \bigabs{ a(z_h,\vphi_h) - a_{h}(Z_h, \phi_h) } & \leq ch^{2} \Lnorm[\surface(t)][p]{\nbg z_h} \Lnorm[\surface(t)][q]{\nbg \vphi_h} , \\
        \bigabs{ g(v_{h};z_h,\vphi_h) - g_{h}(V_{h}; Z_h, \phi_h) }  & \leq ch^{2} \Lnorm[\surface(t)][p]{z_h} \Lnorm[\surface(t)][q]{\vphi_h}, \\
        \bigabs{ b(v_{h};z_h,\vphi_h) - b_{h}(V_{h}; Z_h, \phi_h) } & \leq ch^{2} \Lnorm[\surface(t)][p]{\nbg z_h} \Lnorm[\surface(t)][q]{\nbg \vphi_h},
    \end{align*}
    where the constant $c>0$ is independent from $t\in[0,T]$ and the mesh width $h$.
\end{lemma}
\begin{proof}
These geometric estimates were established for the case $p=q=2$ in \cite[Lemma 5.5]{DziukElliott_L2} and \cite[Lemma 7.5]{LubichMansour_wave}. To show the estimates for general $p$ and $q$, the same proof apply, except the last step where we use a H\"{o}lder inequality.
\end{proof}

\subsection{Weighted norms and basic estimates}
\label{sec:weight-norms-estim}

Similarly, as in the works of Nitsche \cite{Nitsche}, weighted Sobolev norms and their properties play a very important and central role. In this section we recall some basic results for them.

\begin{definition}[Weight function]
  For $\gamma>0$ sufficiently big but independent of $t$ and $h$ we set
  \begin{align*}
    \rho\colon [0,\infty) \to [0,\infty),\quad
    \rho^{2} := \rho^{2}(h) := \gamma h^{2}\abs{\log h}.
  \end{align*}
  % \kb{is this necessary? Cancel?}{Let $y=y(t)\in \surface(t)$ be a curve.}
  We define a weight function
  \(\mu = \mu(t;\, . \, )\colon \surface(t) \to \R\)
  % \kb{time dependency notation}{$\mu= \mu(t;.):\surface(t)\to \R$}
  via the formula
  \begin{align}
    \label{eq:12}
    \mu(x):= \mu(x,y) := \abs{x - y}^{2} + \rho^{2}
    \quad \forall x\in \surface(t).
  \end{align}
\end{definition}

The actual choice of \(\gamma\) is going to be clear from the proofs.
% \begin{remark}
%   The actual choice of $\gamma$ is going to be clear from this section.  \par
%   To derive a weak discrete maximum principle in
%   section~\ref{sec:weak-discr-maxim-1},  we are going to define a different
%   weight.  It is essentially the square root of $\mu$ without $\gamma$ and
%   $\abs{\log h}$.
% \end{remark}

\begin{definition}[Weighted norms, \cite{Nitsche} Section 2.]
%  \label{definition:wt_norm}
  Let $\mu$ be a weight function and $\alpha \in \R$. %and let $u\colon \surface(t)\to \R$ be sufficiently regular.
  We define the norms
  \begin{align*}
    \wtL{u}^2
    =&\ \int_\Ga \mu^{-\alpha} |u|^2,\\
    % \wtL{\nbg u}^2
    % =&\ \int_\Ga \mu^{-\alpha} |\nbg u|^2,\\
    % \wtL{\nbg^{2} u}^{2}
    % =&\ \int_\Ga \mu^{-\alpha} |\nbg^{2} u|^2,\\
    \wtH{u}^{2} = \wtL{u}^{2} + \wtL{\nbg u}^{2},
    \quad
     & \wtHk[2][\alpha]{u}^{2} = \wtH{u}^{2} +
       \wtL{\nbg^{2} u}^{2}.
  \end{align*}
\end{definition}

\begin{lemma}
  \label{lemma:w1i_to_wt_H1}
  Let $\dim \surface(t) = 2$. Let $\phi_{h}\in S_{h}(t)$ with corresponding lift $\varphi_{h} \in S_h^l\t$. Then there exist %a \kb{role of $y$?}{$y=y(t,\phi_{h})\in \surface(t)$} and
  constants $c>0$ independent of $t$, $h$
  and $\gamma$ such that
  \begin{align}
    \label{eq:56}
    \Lnorm[\surface(t)][\infty]{\varphi_{h}}
    \leq&\ c h \abs{\log h} \wtL[2]{\varphi_{h}}, \\
    \label{eq:43}
    \Wnorm[\surface(t)][1,\infty]{\varphi_{h}}
    \leq&\ c \gamma \abs{\log h}^{1/2} \wtH[1]{\varphi_{h}}.
  \end{align}
\end{lemma}

\begin{proof}
There is a point $y_{0,h} \in E_{0}\subset \surface_{h}(t)$ such that
\begin{align*}
\Wnorm[\surface_{h}(t)][1,\infty]{\phi_{h}} = \abs{\phi_{h}(y_{0,h})} + \abs{\nbgh \phi_{h}(y_{0,h})}
= \Wnorm[E_{0}][1,\infty]{\phi_{h}}.
\end{align*}
%Choose $y = y_{0,h}^{l}$.
Note that on $E_{0}$ the estimate $\mu_{h}(x_{h}) \leq c \rho^{2}$ holds for $h < h_{0}$, $h_0$ sufficiently small. Then the second bound yields from using inverse inequality (Lemma~\ref{lemma:inverse_estimate}) and \eqref{eq:41}. The bound \eqref{eq:56} is proved using similar arguments.
\end{proof}

% \begin{lemma}
%   \label{lemma:wtL2_to_Li_or_wtH1_to_W1i}
%   Let $\dim \surface(t) = 2$. There exists a sufficiently small $h_{0}>0$ and a constant $c=c(h_{0})>0$ independent of $t$, $h$ and $\gamma$ such that for all
%   $0<h< h_{0}$ and for all sufficiently smooth $u\colon \surface(t)\to \R$ the following estimates hold
%   \begin{gather}
%     \label{eq:33}
%     \wtL[2]{u}
%     \leq c \rho^{-1}
%     \Lnorm[\surface(t)][\infty]{u}, \\
%     \label{eq:36}
%     \wtH[1]{u}
%     % \leq c (\abs{\log \gamma}^{1/2} + \abs{\log h}^{1/2})
%     \leq c \abs{\log \rho}^{1/2}
%     \Wnorm[\surface(t)][1,\infty]{u}.
%   \end{gather}
% \end{lemma}
% \kb{maybe a slightly better(?) version?}{...}
\begin{lemma}
  \label{lemma:wtL2_to_Li_or_wtH1_to_W1i}
  Let $\dim \surface(t) = 2$. Let $u\colon \surface(t)\to \R$ be a  sufficiently smooth function. Then the  following estimates hold, with a sufficiently small $h_{0}>0$,
  \begin{align}
    \label{eq:33}
    \wtL[2]{u}
    \leq&\ c \rho^{-1}
    \Lnorm[\surface(t)][\infty]{u}, \\
    \label{eq:36}
    \wtH[1]{u}
    % \leq c (\abs{\log \gamma}^{1/2} + \abs{\log h}^{1/2})
    \leq&\ c \abs{\log \rho}^{1/2}
    \Wnorm[\surface(t)][1,\infty]{u},
  \end{align}
  for $0<h< h_{0}$, where the constant $c=c(h_{0})>0$ is independent of $t$, $h$ and $\gamma$.
\end{lemma}

\begin{proof}
  For $\alpha = 1$ or $2$ we obviously have
  % \kb{maybe write $\Lnorm[\surface(t)][\infty]{u} \|\mu^{-\alpha}\|_{L^1(\Gat)}^{1/2}$ ?}{...}
  \begin{equation*}
    \wtL{u}
    \leq \Lnorm[\surface(t)][\infty]{u}
    \Lnorm[\surface(t)][\infty]{u} \|\mu^{-\alpha}\|_{L^1(\Gat)}^{1/2}.
    % \sqrt{\int_{\surface(t)} \mu^{-\alpha}}.
  \end{equation*}
  Then a straightforward calculation, using Appendix~\ref{sec:extens-some-auxil}
  shows both estimates.
\end{proof}

Naturally, there is a weighted version of the Cauchy--Schwarz inequality, namely we have
\begin{equation}
  \label{eq: weighted C--S estimate}
  \begin{split}
    \abs{\aast(z_h,\vphi_h)} & \leq \wtH{z_h} \wtH[-\alpha]{\vphi_h}, \\
    \abs{\aast_{h}(Z_{h},\phi_{h})}
    & \leq c\wtH{z_h} \wtH[-\alpha]{\vphi_h},
  \end{split}
\end{equation}
and similarly for the bilinear forms $g$ and $b$. Furthermore, this yields a weighted version of the geometric errors of the bilinear forms (Lemma~\ref{lemma: geometric errors of forms - Holder}).
%\begin{equation}
%  \label{eq:48}
%  \begin{split}
%    \abs{(g+b)(v; u,w)}
%    & \leq c \Wnorm[\surface(t)][1,\infty]{v}
%    \wtH{u} \wtH[-\alpha]{w}, \\
%    \abs{(g_{h}+b_{h})(v_{h};u_{h},w_{h})}
%    & \leq c \Wnorm[\surface_{h}(t)][1,\infty]{v_{h}}
%    \wtH{u_{h}^{l}} \wtH[-\alpha]{w_{h}^{l}}.
%  \end{split}
%\end{equation}

\begin{lemma}
  \label{lemma:properties_of_wt_norm}
  The following estimates hold, with a constant $c>0$ independent of $t$, $h$ and $\gamma$,
  \begin{align}
    \label{eq:61}
    \abs{\aast(z_h^l,\phi_h^l) - \aast_{h}(Z_h,\phi_h)}
    \leq&\ c h^{2} \wtH{z_h^l} \wtH[-\alpha]{\phi_h^l}, \\
    \lvert (g+b)(v_{h}; z_h^l,\phi_h^l)- (g_{h}+ b_{h})(V_{h}; Z_h,\phi_h) \rvert %\nonumber\\
    \label{eq:62} \leq&\ ch^{2} \wtH{z_h^l} \wtH[-\alpha]{\phi_h^l}.
  \end{align}
\end{lemma}
%\begin{proof}
%This is a consequence of Lemma~\ref{lemma: geometric errors of forms - Holder}.
%\end{proof}

\begin{lemma}
  \label{lemma:calculations_with_weight}
  \begin{enumerate}[label=(\roman*)]
  \item\label{item:4} Derivatives of $\mu^{-1}$ are bounded as
    \begin{equation}
      \label{eq:13}
      \abs{\nbg \mu^{-1}} \leq 2 \mu^{-1,5},\quad
      \abs{\lb \mu^{-1}} \leq c \mu^{-2}
    \end{equation}
    with $c>0$ independent of $t$, $h$ and $\gamma$.
  \item \label{item:6} For arbitrary \(u\in H^{1}\bigl(\surface(t)\bigr)\) the following norm inequalities hold:
    \begin{align}
      \label{eq:44}
      \wtH[-1]{\mu^{-1}u}
      \leq&\ c ( \wtL[2]{u} + \wtH[1]{u}), \\
      \label{eq:51}
      \wtL[-1]{ \mu^{-2}u}
      \leq&\ \rho^{-1} \wtL[2]{u}.
    \end{align}
  \end{enumerate}
\end{lemma}

\begin{proof}
%\ref{item:4}:  It readily follows
%\begin{equation*}
%\abs{\nbg \mu^{-1}}\leq \abs{\nabla \mu^{-1}} \leq
%\frac{2\abs{x-y}}{\mu^{2}}.
%\end{equation*}
%Use $\abs{x-y} \leq \sqrt{\mu}$ to deduce the first part of
%\eqref{eq:13}.
\ref{item:4}:  The first esitmate follows from
\begin{equation*}
\abs{\nbg \mu^{-1}}\leq \abs{\nabla \mu^{-1}} \leq
\frac{2\abs{x-y}}{\mu^{2}} \leq \frac{2\sqrt{\mu}}{\mu^{2}}.
\end{equation*}
For the second inequality consider the formula, % \kb{missing reference!}{cf.~\cite{...}}:
\begin{equation*}
%\label{eq:16}
\lb f = \Delta \bar{f} - \nabla^{2}\bar{f} (\nu, \nu) - H \nu \cdot
\nabla \bar{f},
\end{equation*}
where %$f\colon \surface(t)\to \R$ is sufficiently smooth,
$\bar f \colon U\to \R$ is an extension of the sufficiently smooth function $f$ to an open neighborhood
$U\subset \R^{m+1}$ of $\surface(t)$, $\nabla^{2}\bar{f}$ denotes the
Hessian of $\bar{f}$ and $H$ denotes the trace of the Weingarten map of
$\surface(t)$.  \par
\ref{item:6} In order to show these estimates we use the bounds \eqref{eq:13} obtained above.
\end{proof}

\subsection{Interpolation and inverse estimates}
\label{sec:interp-inverse-estim}

Here we collect some results involving evolving surface finite element functions.

For a sufficiently regular function \(u\colon \surface(t)\to \R\)
we denote by \(\ipol u \in S_{h}(t)\)
its interpolation on \(\surface_{h}(t)\).
Then the finite element interpolation is given by
$\ipoll u = ( \ipol u )^{l}\in S_{h}^{l}(t)$, having the error estimate below, cf.\ \cite{DziukElliott_acta}.
\begin{lemma}
\label{lemma: interpolation error}
    For $m \leq 3$ and \(p\in \{2,\infty\}\), there exists a constant $c>0$ independent of $h$ and $t$ such that for $u\in W^{2,p}\bigl(\Ga(t)\bigr)$:
    \begin{align*}
      \lVert u - \ipoll u \rVert_{L^{p}(\Gat)}
      & + h \lVert \nbg( u - \ipoll u) \rVert_{L^{p}(\Gat)} \\
      & \leq c h^2 \Big( \|\nbg^2 u\|_{L^{p}(\Gat)} + h\|\nbg u\|_{L^{p}(\Gat)} \Big).
    \end{align*}
\end{lemma}

The interpolation estimates hold also if weighted norms are considered.
\begin{lemma}
  \label{lemma:wt_ipol_error_estimate}
  There exists a constant $c>0$ such that for $u\in
  W^{2,\infty}\bigl(\surface(t)\bigr)$ it holds
  \begin{equation}
    \label{eq:57}
    \wtL[2]{u-\ipoll u}^{2} + \wtH[1]{u-\ipoll u}^{2}
    \leq c h^{2} \abs{\log h} \Wnorm[\surface(t)][2,\infty]{u}^{2}.
  \end{equation}
\end{lemma}

\begin{proof}
  Use a H\"{o}lder inequality, Lemma~\ref{lemma: interpolation error} and
  Lemma~\ref{lemma:wtL2_to_Li_or_wtH1_to_W1i}~\eqref{eq:33}, \eqref{eq:36} with
  the choice $u \equiv 1$.
\end{proof}

\begin{lemma}
  \label{lemma: weighted norm estimates}
  There exists $h_{0} >0$, $\gamma_{0} > 0$ such that for all $\alpha \in \R$
  there exists a constant $c=c(h_{0},\gamma_{0})>0$ independent of $t$
  and $h$ such that for all $\gamma > \gamma_{0}$ for the weight $\mu$,
  c.f. \eqref{eq:12}, and for all $h< h_{0}$ the following inequalities holds:
  \begin{enumerate}[label=(\roman*)]
  \item\label{item:1} Let $u\in H^{1}\bigl( \surface(t) \bigr)$ be
    curved element-wise $H^{2}$.  The interpolation
    $\ipoll u \in S_h^{l}\t$ satisfies
    \begin{align}
      \label{eq:17}
        \wtL{u-\ipoll u} + h \wtL{\nbg(u -\ipoll u)} \leq c h^2 \bigl(\wtL{\nbg^{2} u} + ch\wtL{\nbg u}\bigr),
    \end{align}
    where $\wtL{\nbg^{2} u}$ is understood curved element-wise.
  \item\label{item:2} For any $\vphi_h \in S_h^l\t$ the following
    estimate holds:
    \begin{equation}
      \label{eq:18}
%      \begin{split}
        \wtH[-1]{\mu^{-1}\vphi_h - \ipoll(\mu^{-1}\vphi_h)} %\\&
        \leq c \biggl( \frac{h}{\rho} + h\biggr) \bigl(
        \wtL[2]{\vphi_h} + \wtL[1]{ \nbg \vphi_h} \bigr).
%      \end{split}
    \end{equation}
  \end{enumerate}
\end{lemma}
\begin{proof}
  \ref{item:1}: To prove inequality~\eqref{eq:17} it suffices to show that there
  exists a constant $c=c(\alpha)>0$ independent of $t,h$ such that for each
  element $K \in \mathcal{T}_{h}(t)$ it holds
  \begin{align*}
    \int_{K^{l}} \mu^{\alpha}\bigl( (w-\ipoll w)^{2}
    &+ h \abs{\nbg (w- \ipoll w)}^{2} \bigr) \\
    &\leq ch^{2} \int_{K^{l}} \mu^{\alpha}\bigl( \abs{\nbg^{2} w}^{2} + ch
      \abs{\nbg w}^{2}\bigr),
  \end{align*}
  where $K^{l} \subset \surface(t)$ denote the lifted curved element of $K$.
  It is easy to show that there exists $\gamma_{0}= \gamma_{0}(h_{0}) >0$
  and $c= c(\gamma_{0})>0$ such that for all $\gamma > \gamma_{0}$ it holds
  \begin{align*}
    \max_{K \in \mathcal{T}_{h}} \biggl( \frac{\max_{x\in K^{l}}
    \mu(x,y)}{\min_{x\in K^{l}} \mu(x,y)}\biggr) \leq c.
  \end{align*}
  % ------------------------------------------------------------
  % Detailed proof
  % Let $K\in\pi_h$. Then the length of a curve $\gamma$ on the \co\ surface
  % between two points $p_1, p_2 \in K$ can be estimated as
  % \begin{align*}
  %   \riemdist(p_1^l,p_2^l) & \leq \int_0^1 \sqrt{ |\gamma'\t|^2 } \d t
  %   = \int_0^1 {\big|\pr(I-d\wein)(p_2-p_1)\big|} \d t \\
  %   &= \int_0^1 {\big|I-\nu\nu^T-d\wein\big| \ \big|p_2-p_1\big|} \d t \\
  %   & \leq (1+c_0h^2)|p_2-p_1| \leq (1 + c_0 h^2 ) c_2 h \leq C h,
  % \end{align*}
  % where $C$ is independent of $h$.  Then, since there is a point $p^l \in K^l$
  % such that $\riemdist(z,K^l)=\riemdist(z,p^l)=L$, we can estimate (for a fixed
  % triangle) as
  % \begin{alignat*}{2}
  %   \frac{\sigma_{z}(x)}{\sigma_{z}(x)} &\leq \frac{\sigma_{z}(x)}{L^2 +
  %       \kappa^2h^2}
  %   \leq \frac{(L+Ch)^2 + \kappa^2h^2}{L^2 + \kappa^2h^2}\\
  %   &=    \frac{ L^2 + 2L Ch + C^2h^2 + \kappa^2h^2}{L^2 + \kappa^2h^2}
  %   \leq \frac{ L^2 + C^2 + h^2 + C^2h^2 + \kappa^2h^2}{L^2 +
  %       \kappa^2h^2} \\
  %   &\leq 1+\frac{ L^2 + Ch^2 + C^2h^2}{L^2 + \kappa^2h^2} \leq c,
  % \end{alignat*}
  % for $\kappa\geq C+C^2$ sufficiently large (but independent of $h,\ x$ and
  % $z$). The last inequality is also independent of the chosen triangle,
  % therefore the assertion is proved.
  % ------------------------------------------------------------
  A straightforward calculation finishes the proof.  \par
  % ------------------------------------------------------------
  % Detailed proof
  % Consider the case $-\alpha < 0$.
  % \begin{align*}
  %   \int_{K^{l}} \mu^{-\alpha}\bigl[ (v-\ipoll v)^{2}
  %   & + h\abs{\nbg(v - \ipoll v)}^{2}\bigr]  \\
  %   & \leq (\min_{K^{l} }\mu)^{-\alpha} \int_{K^{l}} \bigl[ (v-\ipoll v)^{2}
  %     + h \abs{\nbg( v- \ipoll v)}^{2}\bigr] \\
  %   & \leq (\min_{K^{l} }\mu)^{-\alpha}
  %     \int_{K^{l}}\frac{\mu^{-\alpha}}{\mu^{-\alpha}} \bigl[ (v-\ipoll v)^{2}
  %     + h \abs{\nbg( v- \ipoll v)}^{2}\bigr] \\
  %   & \leq (\min_{K^{l} }\mu)^{-\alpha} \frac{1}{(\max_{K^{l}}
  %     \mu)^{-\alpha}} \int_{K^{l}} [\ldots] \\
  %   & \leq \biggl( \frac{\max_{K^{l}} \mu}{\min_{K^{l}}
  %     \mu}\biggr)^{\alpha} \int_{K^{l}} [\ldots]
  % \end{align*}
  % For the case $- \alpha >0$ do first max and then min.
  % ------------------------------------------------------------
  \ref{item:2}: For an arbitrary function $f\colon \surface_{h}(t)\to \R$, which
  is element-wise $H^{2}$, a short calculation, similar to the one done in Dziuk
  \cite[Lemma~3]{Dziuk88}, shows that
  \[
    \abs{(\nbg)_{i}(\nbg)_{j} (f^{l})} \leq c \bigl( \abs{((\nbgh)_{i}
      (\nbgh)_{j}f)^{l}} + c h \abs{\nbg (f^{l})}\bigr),
  \]
  for a sufficiently small $h_{0} > h>0$.  A straightforward calculation combined with
  \ref{item:1} and \eqref{eq:13} shows the claim.
  % ----------------------------------------------------------------------
  % Detailed calculation:
  % \begin{align*}
  %   \wtH[-1]{\mu^{-1}\varphi_{h}^{l} &- \ipoll (\mu^{-1}\varphi_{h}^{l})} \\
  %   & \leq  c \Bigl( h \wtL[-1]{\nbg^{2}(\mu^{-1}\varphi_{h}^{l})} + h^{2}
  %   \wtL[-1]{\nbg(\mu^{-1}\varphi_{h}^{l} )}\Bigr).
  % \end{align*}
  % The second term:
  % \begin{align*}
  %   \wtL[-1]{\nbg(\mu^{-1}\varphi_{h}^{l} )}
  %   & \leq c \Bigl( \wtL[2]{\varphi_{h}^{l}} +
  %     \wtL[1]{\nbg \varphi_{h}^{l}}\Bigr)
  % \end{align*}
  % Even without the $h^{2}$ this yields a good term.  The first term:
  % \begin{align*}
  %   \wtL[-1]{\nbg^{2}(\mu^{-1}\varphi_{h}^{l})}
  %   & \leq c\Bigl( \wtL[3]{\varphi_{h}^{l}} +
  %     \wtL[2]{\nbg\varphi_{h}^{l}} + h
  %     \wtL[1]{\nbg\varphi_{h}^{l}}\Bigr)\\
  %   & = c\Bigl( \frac{1}{\sqrt{\mu}} \wtL[2]{\varphi_{h}^{l}}
  %     + (\frac{1}{\sqrt{\mu}} + h) \wtL[1]{\nbg
  %     \varphi_{h}^{l}}\Bigr).  \qedhere
  % \end{align*}
  % ----------------------------------------------------------------------
\end{proof}

The following general version of inverse estimates for finite element functions plays a key role later on, cf.\ \cite{1980_stw}.
\begin{lemma}[Inverse estimate]
  \label{lemma:inverse_estimate}
  There exists $c>0$ such that for each triangle $E_{h}\t \subset
  \surface_{h}(t)$ the following inequality holds
  \[
  \Wnorm[E_{h}\t][k,p]{\varphi_{h}\t} \leq c h^{m-k-2(1/q-1/p)}
  \Wnorm[E_{h}\t][m,q]{\varphi_{h}\t} \qquad (\forall\varphi_{h} \in
  S_{h}(t)) .
  \]
\end{lemma}

\begin{lemma}%[Improved inverse estimate]
  \label{lemma:improved_inverse_estimate}
  There exists $c>0$ with
  \begin{equation*}
%    \Lnorm[\surface(t)][\infty]{\phi^{l}_{h}} \leq \frac{1}{V} \Big\lvert \int_{\surface(t)}
%    \phi^{l}_{h}(y) \mathrm{d}V(y) \Big\rvert + c \abs{\log
%    h}^{1/2} \Lnorm[\surface(t)][2]{\nbg \phi^{l}_{h}}.
    \Lnorm[\surface(t)][\infty]{\vphi_{h}} \leq \biggl\lvert
    \frac{1}{V}\int_{\surface(t)} \vphi_{h}(y) \mathrm{d}V(y)\biggr\rvert + c
    \abs{\log h}^{1/2} \Lnorm[\surface(t)][\infty]{\nbg
    \vphi_{h}}.
  \end{equation*}
\end{lemma}
\begin{proof}
Follow the steps in Schatz, Thom\'{e}e, Wahlbin \cite{1980_stw} using the
Green's function
from Theorem~\ref{thm:green} and calculating with geodesic polar coordinates.
\end{proof}

%By $\ipoll\colon H^{1}\bigl(\Ga (t)\bigr) \to S_{h}^{l}(t)$ we denote the usual finite element interpolation operator, having the error estimate below.
%\begin{lemma}[\cite{DziukElliott_acta}]
%\label{lemma: interpolation error}
%    For $m \leq 3$, there exists a constant $c>0$ independent of $h$ and $t$ such that for $u\in H^{2}\bigl(\Ga(t)\bigr)$:
%    \begin{align*}
%      \lVert u - \ipoll  u \rVert_{L^{2}(\Gat)}
%      & + h \lVert \nbg( u - \ipoll  u) \rVert_{L^{2}(\Gat)} \\
%      & \leq\ c h^2 \Big( \|\nbg^2 u\|_{L^2(\Gat)} + h\|\nbg u\|_{L^2(\Gat)} \Big).
%%        \lVert u - \ipoll  u \rVert_{L^{2}(\Gat)} + h \lVert \nbg( u - \ipoll  u) \rVert_{L^{2}(\Gat)} \leq&\ c h^{2} \|u\|_{H^{2}(\Gat)}.
%    \end{align*}
%\end{lemma}

\subsection{Estimates for an $L^2$-projection}
\label{sec:estimates-an-l2}

This section shows some technical results for the $L^2$-projection, which is denoted by $\Lproj$ (in contrast with the Ritz map which will be denoted by $\lRmap$).

\begin{definition}[$L^{2}$-projection]
  \label{definition:L2_proj}
  We define $\Lproj(t)\colon L^{2}\bigl(\surface_{h}(t)\bigr) \to S_{h}(t)$ as
  follows: Let $u_{h} \in L^{2}\bigl(\surface_{h}(t)\bigr)$ be given.  Then
  there exits a unique finite element function $\Lproj(t)u\in S_{h}(t)$ such
  that for all $\phi_{h}\in S_{h}(t)$ it holds
  \begin{align}
    \label{eq:47}
    m_{h}\bigl( \Lproj(t)u_{h}, \phi_{h}\bigr)
    = m_{h}( u_{h}, \varphi_{h}).
  \end{align}
\end{definition}

The following important $L^{p}$-stability bound and exponential decay property
from Douglas, Dupont and Wahlbin \cite[equation~(6) and (7)]{1975_ddw} holds
without any serious modification.

\begin{theorem}
  \label{theorem:old_L2_bound}
  For $p\in [1,\infty]$ let $u_{h}\in L^{p}\bigl(\surface_{h}(t)\bigr)$.  Then
  there exists a constant $c>0$ independent of $h$ and $t$ such that
  \begin{equation*}
    \Lnorm[\surface_{h}(t)][p]{\Lproj(t)u_{h}}
    \leq c \Lnorm[\surface_{h}(t)][p]{u_{h}}.
  \end{equation*}
  Further there exists $c_{2},c_{3}>0$ independent of \(h\)
  and \(t\) such that for $A_{h}^{1}(t)$ and $A_{h}^{2}(t)$ disjoint subsets of
  $\surface_{h}(t)$ with $\supp(u_{h}) \subseteq A_{h}^{1}$ we have
  \begin{equation}
    \label{eq:14}
    \Lnorm[A_{h}^{2}(t)]{\Lproj(t) u_{h}} \leq c_{2} e^{-c_{3}
      \dist_{h}(A^{1}_{h},A^{2}_{h})h^{-1}} \Lnorm[A_{h}^{1}(t)]{u_{h}},
  \end{equation}
  where $\dist_h(x,y)=\dist_{\surface_{h}(t)}(x,y)$ is the intrinsic Riemannian
  distance of $\surface_{h}(t)$.
\end{theorem}

% An easy consequence, which uses the projection property,
% $\Lproj(t) \phi_{h} = \phi_{h}$, is the following corollary.

% \begin{corollary}
%   \label{corollary:old_L2_eror}
%   Let $p,u_{h}$ be like in Theorem~\ref{theorem:old_L2_bound}.  Then it holds
%   \begin{align*}
%     \Lnorm[\surface_{h}(t)][p]{u_{h} - \Lproj(t)u_{h}}
%     \leq c \inf_{\phi_{h}\in S_{h}(t)} \Lnorm[\surface_{h}(t)][p]{u_{h} -
%     \phi_{h}}.
%   \end{align*}
% \end{corollary}

For the proof of our discrete weak maximum principle we are going to use a
different weight function then \eqref{eq:12}.  Let \([0,T]\to \R^{m+1}, t
\mapsto y(t)\) be a curve with the property \(y(t)\in \surface(t)\).  In the
following we write \(y\) instead of \(y(t)\).  We define
\begin{equation}
  \label{eq:10}
  \pwt(x) := \pwt^{y}(x) := \pwt(x,y):=  \bigl( \abs{x-y}^2 +
  h^{2}\bigr)^{1/2}.
\end{equation}
We gather some estimates concerning \(\pwt\) in the next lemma.

\begin{lemma}
  \label{lemma:rho_calculation}
  There exists a constant $c>0$ independent of \(t\) and \(h\) such that the
  following estimates hold
  % for a fixed $t\in [0,T]$ and
  % $x\in \surface(t)$ the following estimates holds for $\rho = \rho^{x}$:
  \begin{gather}
    \label{eq:67}
    \Lnorm[\surface(t)][\infty]{\dellmat \pwt}
    \leq c, \qquad
    \Lnorm[\surface(t)][\infty]{\dellmat_{h} \pwt}
    \leq c,\\
    \label{eq:68}
    \Lnorm[\surface(t)][\infty]{ \nbg \pwt} \leq 1,\qquad
    \abs{ \nbg^{2} \pwt} \leq c \biggl(\frac{1}{\pwt} + 1 \biggr), \qquad
    \Lnorm[\surface(t)][\infty]{ \nbg^{2} (\pwt^{2})} \leq c.
  \end{gather}
  % For $\pwt^{-l}$ we have the estimates:
  % \begin{align*}
  %   \Lnorm[\surface_{h}(t)][\infty]{\dellmat_{h} \pwt^{-l}} &\leq \Lnorm[\surface_{h}(t)][\infty]{V_{h}}, \qquad
  %   \Lnorm[\surface_{h}(t)][\infty]{ \nbgh \pwt^{-l}} \leq c.
  % \end{align*}
  % For each triangle $E_h\subset \surface_{h}(t)$ we have the estimate
  % \begin{align*}
  %   \abs{ \nbgh^{2} \pwt^{-l}}
  %   &\leq c \biggl(\frac{1}{\pwt^{-l}} +h + h^{2}\biggr), \qquad
  %   \Lnorm[E_h][\infty]{ \nbgh^{2} \big((\pwt^{2})^{-l}\big)} \leq c.
  % \end{align*}
\end{lemma}

The proof of this lemma is a straightforward calculation and is omitted here.

\begin{lemma}
  \label{lemma:weighted_l_two_proj_err}
  There exists $c>0$ such for fixed $t\in [0,T]$, $x_{h}\in \surface_{h}(t)$,
  $\pwt = \pwt^{x_{h}}$,  $\phi_h \in S_h (t)$ and $\psi_h :=
  \Lproj ( \pwt^{2} \varphi_{h})$ the following inequality holds:
  \begin{multline*}
    \Lnorm[\surface_{h}(t)]{\pwt^{2}\phi_{h} - \psi_{h}} + h
    \Lnorm[\surface_{h}(t)]{\nbgh(\pwt^{2}\phi_{h} - \psi_{h})} \\
    \leq c h^{2}\big(\Lnorm[\surface_{h}(t)]{\phi_{h}}
    + \Lnorm[\surface_{h}(t)]{\pwt \nbgh\phi_{h}} \big).
  \end{multline*}
\end{lemma}

\begin{proof}
  Consider a triangle $E_h \subset \surface_{h}(t)$ and set $g_{h}:=
  \ipol (\pwt^{2}\phi_{h})$. Use
  Lemma~\ref{lemma:rho_calculation} and \eqref{eq:42}
  % \ref{lemma: rho equivalence}
  and follow the steps in Schatz, Thom\'{e}e and Wahlbin
  \cite[Lemma~1.4]{1980_stw}.
\end{proof}

\section{A Ritz map and some error estimates}
% \section{A Ritz map for evolving surface problems}
\label{sec:ritz-map-some}

Just as in the usual $L^2$-theory the Ritz map plays a very important role for our
$L^\infty$-error estimates. This section is devoted to the careful $L^\infty$- and weighted norm analysis of the errors in the Ritz map.

\begin{definition}[Ritz map, \cite{LubichMansour_wave}]
%\label{definition:modified_Ritz_map}
  We define $\Rmap(t)\colon H^{1}\bigl(\surface(t)\bigr) \to S_{h}(t)$ as
  follows: Let $u \in H^{1}\bigl(\surface(t)\bigr)$ be given.  Then there exits
  a unique finite element function $\Rmap(t)u\in S_{h}(t)$ such that for all
  $\phi_{h}\in S_{h}(t)$ with $\varphi_{h}= \phi_{h}^{l}$ it holds
  \begin{align}
    \label{eq:39}
    \aast_{h}\bigl( \Rmap(t)u, \phi_{h}\bigr)
    = \aast( u, \varphi_{h}).
  \end{align}
  This naturally defines the Ritz map on the continuous surface:
  \begin{equation*}
    \lRmap(t)u = \big(\Rmap(t)u\big)^{l}  \in S_{h}^{l}(t) .
  \end{equation*}
\end{definition}
Note that the Ritz map does not satisfy the Galerkin orthogonality, however it satisfies, using \eqref{eq:61}, the following estimate, cf.~\cite{LubichMansour_wave}. For all $\varphi_{h}\in S_{h}^{l}(t)$ we have
\begin{equation}
  \label{eq:49}
  \abs{\aast\bigl( u - \lRmap(t)u, \varphi_{h}\bigr)}
  \leq ch^{2} \wtH{\lRmap(t)u} \wtH[-\alpha]{\varphi_{h}} .
\end{equation}

In this section we aim to bound the following errors of the Ritz map:
\begin{align*}
  u - \lRmap(t)u
  \quad \text{and}
  \quad
  \mat_{h}\bigl( u -
  \lRmap(t)u \bigr),
\end{align*}
in the $L^{\infty}$- and $W^{1,\infty}$-norms.  Previously, $H^{1}$- and
$L^{2}$-error estimates have been shown in \cite{DziukElliott_ESFEM,DziukElliott_L2}.

% \bigskip \par
% $\Rmap(t)$ is not necessary a projection, e.g.\ for $\phi_{h}\in S_{h}(t)$ with
% $\varphi_{h} := \phi_{h}^{l}$ in general we have
% $\Rmap(t)\varphi_{h} \neq \phi_{h}$.  Since the work of Nitsche \cite{Nitsche}
% rely on this fact, we introduce for our evolving surface problem an auxiliary
% Ritz projection.
% \begin{definition}[Ritz projection]
%   We define $\Rproj(t)\colon H^{1}\bigl(\surface(t)\bigr) \to S_{h}^{l}(t)$ as
%   follows: Let $u \in H^{1}\bigl(\surface(t)\bigr)$ be given.  Then there exits
%   a unique lifted finite element function $\Rproj(t)u\in S_{h}^{l}(t)$ such that
%   for all $\varphi_{h}\in S_{h}^{l}(t)$ it holds
%   \begin{align}
%     \label{eq:9}
%     \aast\bigl( \Rproj(t)u, \varphi_{h}\bigr)
%     = \aast( u, \varphi_{h}).
%   \end{align}
% \end{definition}
% Our strategy is to first bound the quantities
% \begin{align*}
%   \Lnorm[\surface(t)][\infty]{u - \Rproj(t)u}
%   \quad \text{and} \quad
%   \Wnorm[\surface(t)][1,\infty]{u - \Rproj(t)u},
% \end{align*}
% by generalizing Nitsches weighted norm approach and his Ritz projection
% stability theorem.

\subsection{Weighted a priori estimates}
%\label{sec:weight-norms-auxill}

Before turning to the maximum norm error estimates, we state and prove some technical regularity results involving weighted norms.

\begin{lemma}[Weighted a priori estimates]
\label{lemma:wt_ellliptic}
  For $f\in L^{2}\bigl(\surface(t)\bigr)$, the problem
  \begin{align*}
    -\lb w + w = f \qquad \textrm{ on } \, \surface(t),
  \end{align*}
  has a unique weak solution $w\in H^{1}(\surface(t))$. Furthermore, $w\in
  H^{2}\bigl(\surface(t)\bigr)$ and we have the following weighted a priori
  estimates
  \begin{align}
    \label{eq:26}
    \wtH[-1]{w}
    &\leq c \bigl( \wtL[-1]{f} + \Lnorm{w} \bigr)\\
    \label{eq:24}
    \wtHk[2][-1]{w}
    &\leq c\bigl( \wtL[-1]{f} + \Hnorm{w}\bigr),
  \end{align}
  where the constant $c>0$ is independent of $t,h$ and $\gamma$.
\end{lemma}

\begin{proof}
  Existence and uniqueness of a weak solution follows from \cite{Aubinbook}.
  Using integration by parts, Young inequality and
  $\abs{\nbg \mu} \leq \sqrt{\mu}$ a short calculation shows \eqref{eq:26}.
  %\kb{something was missing!?}{and then the a priori estimate below}.
  % ------------------------------------------------------------
  % Detailed calculation
  % \begin{align*}
  %   \wtH[-1]{w}^{2}
  %   & = \int \mu \bigl( w^{2} + \abs{\nbg w}^{2}\bigr)
  %     = \int \mu f w - w \nbg \mu \cdot \nbg w.
  % \end{align*}
  % Observe that $\abs{\nbg \mu} \leq \abs{x - y}$.  This lead to
  % \begin{align*}
  %   \wtH[-1]{w}^{2}
  %   & \leq \frac{1}{2} \wtL[-1]{f}^{2} + \frac{1}{2} \wtL[-1]{w}^{2}
  %     + \frac{1}{2} \wtL[-1]{\nbg w} + \frac{1}{2} \Lnorm{w}.
  % \end{align*}
  % ------------------------------------------------------------
  For the details on elliptic regularity and a derivation of the a
  priori estimate
  \[
    \Hnorm[][2]{w} \leq c \Lnorm{ - \lb w + w},
  \]
  where $c>0$ is independent of $t$, we refer to \cite[Appendix~A]{2014nonlinear}.

  Because of \eqref{eq:26} it suffices to
  prove \eqref{eq:24} for $\wtL[-1]{\nbg^{2}w}^{2}$ as the left-hand side instead
  of $\wtHk[2][-1]{w}^{2}$.  Apply the usual elliptic a priori estimate on
  $(x^{i}-y^{i}) w$ for $i=1,\dotsc, m+1$ to get the desired estimate.
  % ------------------------------------------------------------
  % Detailed calculation
  % \\
  % It holds
  % \begin{align*}
  %   \mu \abs{\nbg^{2} w}^{2}
  %   & = \sum_{i=1}^{m+1} \abs{ (x^{i}-y^{i})
  %     \nbg^{2}w}^{2} + \abs{ \rho  \nbg^{2}w }^{2}\\
  %   & = \sum_{i=1}^{m+1} \abs{ \nbg^{2}\bigl( (x-y)^{i} w\bigr) -  2 e^{i}
  %     \odot \nbg w  }^{2} + [\ldots]
  % \end{align*}
  % From that we deduce
  % \begin{align*}
  %   \wtL[-1]{\nbg^{2} w}
  %   & \leq \sum_{i=1}^{m+1}c \sqrt{ \int_{\surface(t)} \bigl( (-\lb + I)
  %     [(x-y)^{i}w]\bigr)^{2}} \\
  %   & \quad + c\sqrt{2(m+1)} \Lnorm{\nbg w} + c\wtL[-1]{f}^{2}
  % \end{align*}
  % For $L := -\lb + I$ we have the chain rule
  % \begin{align*}
  %   L(fg) = L(f) g - f \lb g - 2 \nbg f \cdot \nbg g.
  % \end{align*}
  % Note that $-\lb (x^{i}-y^{i}) = H \nu^{i}$ is.  From that it follows
  % \begin{align*}
  %   \sum_{i=1}^{m+1}c \sqrt{\int_{\surface(t)} \bigl( (-\lb + I)
  %   [(x-y)^{i}w]\bigr)^{2}}
  %   & \leq c \Bigl( \wtL[-1]{f}^{2} + \Lnorm{w}^{2}+ \Lnorm{\nbg w}^{2} \Bigr).
  % \end{align*}
  % Putting everything together implies the claim.
  % ------------------------------------------------------------
\end{proof}

\begin{lemma}
  \label{lemma:wt_eigenvalues}
  For $g\in L^{2}\bigl(\surface(t)\bigr)$ the problem
  \begin{align*}
    - \lb w + w = \mu^{-2} g.
  \end{align*}
  has a unique weak solution  $w\in H^{1}(\surface(t))$. Furthermore, $w\in
  H^{2}\bigl(\surface(t)\bigr)$, and there exists a constant $c>0$ independent of $t$ and $h$ such that
  \begin{align}
    \label{eq:45}
    \Hnorm{w}^{2} & \leq c \rho^{-2} \abs{\log \rho} \wtL[2]{g}^{2}.
  \end{align}
\end{lemma}

\begin{proof}
  Lemma~\ref{lemma:wt_ellliptic} gives us existence, uniqueness and regularity
  of $w$.
  Consider the number
  \begin{align*}
    \lambda^{-1}(t) :=
    \sup \bigl\{ \Hnorm{f}^{2} \mid f\in H^{2}\big(\surface(t)\bigr),
    \; \wtL[-2]{-\lb f + f}^{2} \leq 1\}.
  \end{align*}
  Inequality~\eqref{eq:45} is proven if we show
  \begin{align*}
    \lambda^{-1}(t) \leq c \rho^{-2} \abs{\log \rho},
  \end{align*}
  where $c$ is $t$ independent.  A short calculation shows that the smallest
  eigenvalue $\widetilde{\lambda}_{\min}(t)$ of the elliptic eigenvalue problem
  \begin{align*}
    - \lb f + f = \widetilde{\lambda} \mu^{-2} f \quad \text{on $\surface(t)$}
  \end{align*}
  is equal to $\lambda(t)$.
  The weighted Rayleight quotient implies
  \begin{align*}
    \widetilde{\lambda}_{\min}
    &= \inf_{f\in H^{1}}\frac{\Hnorm{f}^{2}}{\wtL[2]{f}^{2}}.
  \end{align*}
  Hence it suffices to prove
  \begin{align}
    \label{eq:50}
    \wtL[2]{f}^{2}
    & \leq  c \rho^{-2} \abs{\log(\rho)}\Hnorm{f}^{2},
  \end{align}
  for a $f\in H^{1}$.  With a H\"{o}lder estimate we arrive at
  \begin{align*}
    \wtL[2]{f}^{2}
    & \leq \biggl( \int_{\surface(t)} \mu^{-2p}\biggr)^{1/p}
      \biggl( \int_{\surface(t)} f^{2q}\biggr)^{1/q}
      = \biggl( \int_{\surface(t)} \mu^{-2p}\biggr)^{1/p}
      \Lnorm[\surface(t)][2q]{f}^{2},
  \end{align*}
  where $1< p,q< \infty$ satisfies $p^{-1}+q^{-1}=1$.  We take the choice
  $q= \sqrt{\abs{\log \rho}}$.
  % which implies that $p=  \frac{\abs{\log \rho}}{\abs{\log \rho} - 1}$.
  It is easy to prove the following quantitative Sobolev-Nierenberg inequality
  for moving surfaces:
  \begin{align*}
    \Lnorm[\surface(t)][q]{f} \leq c q \Hnorm[\surface(t)]{f},
  \end{align*}
  where $c$ is independent of $t$ and $q$.  A straightforward calculation with
  geodesic polar coordinates using Lemma~\ref{lemma:comparision_extr_intr_dist}
  and Lemma~\ref{lemma:geodesic_sphere} shows inequality~\eqref{eq:50}.
  % ----------------------------------------------------------------------
  % Details
  % \\ We have the estimate
  % \begin{align*}
  %   \Lnorm[][2q]{f}^{2}
  %   \leq c q^{2} \Hnorm{f}^{2} \leq c \abs{\log\rho} \Hnorm{f}^{2}.
  % \end{align*}
  % For the first term on the right-hand side we calculate
  % \begin{align*}
  %   \int \mu^{-2p}
  %   &\leq c\int_{0}^{\infty}\frac{r}{(r^{2}+ c\rho^{2})^{2p}} \d r \\
  %   & \leq c \frac{1}{2p - 1} (c \rho^{2})^{-2p + 1}.
  % \end{align*}
  % After taking the $p$-th root we get
  % \begin{align*}
  %   \biggl(\int \mu^{-2p}\biggr)^{1/p}
  %   \leq c^{1/p} (2p -1)^{-1/p} c^{(-2p +1) /p} \rho^{2( -2p +1)/p}.
  % \end{align*}
  % We have
  % \begin{align*}
  %   p
  %   & \to 1\\
  %   c^{1/p}
  %   & \to c \\
  %   (2p -1)^{1/p}
  %   & \to 1 \\
  %   c^{(-2p+1)/p}
  %   & \to c \\
  %   \rho^{2(-2p+1)/p}
  %   & = \mathcal{O}(\rho^{-2}).
  % \end{align*}
  % ----------------------------------------------------------------------
\end{proof}

\subsection{Maximum norm error estimates}
%\label{sec:error-estimates-ritz}

Before showing $L^\infty$- and $W^{1,\infty}$-norm error estimates for the Ritz map, we show similar estimates for weighted norms. Then, by connecting the norms, use these results to obtain our original goal.

Throughout this subsection, we write $\lRmap u$ instead of $\lRmap(t) u$.
\begin{lemma}
  \label{lemma:wt_ritz_error_estimate}
  There exists $h_{0}>0$ sufficiently small and $\gamma_{0}>0$ sufficiently
  large and a constant $c=c(h_{0},\gamma_{0})>0$ such that for $u\in
  W^{2,\infty}\bigl(\surface(t)\bigr)$ it holds
  \begin{equation}
    \label{eq:53}
    \wtL[2]{u-\lRmap u}^{2} + \wtH[1]{u-\lRmap u}^{2}
    \leq c h^{2} \abs{\log h} \Wnorm[\surface(t)][2,\infty]{u}^{2}.
  \end{equation}
\end{lemma}

\begin{proof}
  \textbf{Step~1:} Our goal is to show
  \begin{equation}
    \label{eq:46}
    \wtH[1]{u- \lRmap u}^{2}
    \leq  c h^{2}\abs{\log h} \Wnorm[\surface(t)][2,\infty]{u}^{2}
      + \hat c \wtL[2]{u - \lrmap u}^{2} .
  \end{equation}
%  Using \eqref{eq:13} the bound on $\mu$, and by partial integration, similarly as in Nitsche \cite[Theorem~1]{Nitsche}, we deduce
  Similarly as in Nitsche \cite[Theorem~1]{Nitsche}, \eqref{eq:13} and partial
  integration yields
  \begin{align*}
    \wtH[1]{u- \lRmap u}^{2}
    & \leq \aast\bigr( u - \lrmap u, \mu^{-1} (u- \lrmap u)\bigr)
      + c \int_{\surface(t)} (\lb \mu^{-2}) (u - \lrmap u)^{2} \\
    & \leq \aast\bigr( u - \lrmap u, \mu^{-1} (u- \lrmap u)\bigr)
      + c \wtL[2]{u - \lrmap u}^{2}.
  \end{align*}
  For %the convenience of the reader we define
  simplicity we set $e= u - \lrmap u$, and use $\ipoll u=(\ipol u)^{l}$ to obtain
  \begin{align*}
    \aast(e, \mu^{-1} e)
    & = \aast\bigl(e, \mu^{-1}( u -  \ipoll u)\bigr) \\
    & \quad + \aast\Bigl(e, \mu^{-1}(\ipoll u - \lrmap u)
      - \ipoll\bigl( \mu^{-1}(\ipoll u - \lrmap u)\bigr)\Bigr) \\
    & \quad + \aast\Bigl(e, \ipoll\bigl(\mu^{-1}(\ipoll u - \lrmap)\bigr)\Bigr)
%    \\ &
        = I_{1} + I_{2} + I_{3}.
  \end{align*}
  Using Lemma~\ref{lemma:properties_of_wt_norm}~\eqref{eq: weighted C--S
    estimate}, Lemma~\ref{lemma:calculations_with_weight}~\eqref{eq:44},
  Lemma~\ref{lemma:wt_ipol_error_estimate}~\eqref{eq:57} and $\eps$-Young
  inequality we estimate as
  \begin{align*}
    \abs{I_{1}}
    % ------------------------------------------------------------
    % Details
    % & \leq \wtH[1]{e} \wtH[-1]{\mu^{-1}(u - \ipoll u)} \\
    % & \leq c \wtH[1]{e} ( \wtH[1]{u - \ipoll u} + \wtL[2]{u - \ipoll u})  \\
    % & \leq c \wtH[1]{e} ( h \wtHk[2][1]{u} + h^{2} \wtHk[2][2]{u}) \\
    % & \leq c \wtH[1]{e} ( h \abs{\log \gamma}^{1/2} + h \abs{\log h}^{1/2}
    %   + h^{2} \rho^{-1}) \Wnorm[\surface(t)][2,\infty]{u} \\
    % ------------------------------------------------------------
    & \leq \eps \wtH[1]{e}^{2} + c h^{2} \abs{\log h}
      \Wnorm[\surface(t)][2,\infty]{u}^{2} \phantom{\bigl( + \wtL[2]{e}\bigr)}.
  \end{align*}
  For the second term use in addition Lemma~\ref{lemma: weighted norm
    estimates}~\eqref{eq:18} and a $0< h < h_{0}$ sufficiently small to get
  \begin{align*}
    \abs{I_{2}}
    % ------------------------------------------------------------
    % Details
    % & \leq \wtH[1]{e}
    %   \wtH[-1]{\mu^{-1}( \ipoll u - \Rmapl u ) - \ipoll \bigl(
    %   \mu^{-1}(\ipoll u - \Rmapl u)\bigr)} \\
    % & \leq c\wtH[1]{e} \left(\frac{h}{\rho} + h\right)
    %   ( \wtL[2]{\ipoll u - \Rmapl u } + \wtH[1]{\ipoll u - \Rmapl u}) \\
    % & \leq c \wtH[1]{e} \left(\frac{h}{\rho} + h\right)
    %   ( \wtL[2]{u - \ipoll u} + \wtH[1]{u - \ipoll u}
    %   + \wtL[2]{e} + \wtH[1]{e}) \\
    % & \overset{h<h_{0}}{\leq} \frac{1}{4}\wtH[1]{e}
    %   ( c h\abs{\log h}^{1/2} \Wnorm[\surface(t)][2,\infty]{u}
    %   + \wtL[2]{e} + \wtH[1]{e}) \\
    % & \leq \frac{1}{8} \wtH[1]{e}^{2} + \frac{1}{8}  \wtH[1]{e}^{2}
    %   + c \bigl(\wtL[2]{e}^{2}
    %   + h^{2} \abs{\log h} \Wnorm[\surface(t)][2,\infty]{u}^{2}
    %   \bigr) \\
    % ------------------------------------------------------------
    & \leq \eps \wtH[1]{e}^{2} + c \bigl( h^{2} \abs{\log h} \Wnorm[\surface(t)][2,\infty]{u}
      + \wtL[2]{e}\bigr).
  \end{align*}
  For the last term use in addition
  Lemma~\ref{lemma:properties_of_wt_norm}~\eqref{eq:49} to reach at
  \begin{align*}
    \abs{I_{3}}
    % ------------------------------------------------------------
    % Details
    % & \leq ch^{2} \wtH[1]{\Rmapl u} \wtH[-1]{\ipoll \bigl(
    %   \mu^{-1} (\ipoll u - \Rmapl u)\bigr)} \\
    % & \leq ch^{2} (\wtH[1]{u} + \wtH[1]{e}) \\
    % & \quad \left( \wtH[-1]{\mu^{-1}( \ipoll u - \Rmapl u)}
    %   + \wtH[-1]{ \mu^{-1}( \ipoll u - \Rmapl u) -
    %   \ipoll\bigl( \mu^{-1}( \ipoll u - \Rmapl u) \bigr)}\right) \\
    % & \leq ch^{2} ( \wtH[1]{u} + \wtH[1]{e})
    %   (\wtL[2]{\ipoll u - \Rmapl u} + \wtH[1]{\ipoll u - \Rmapl u}) \\
    % & \leq ch^{2} ( \wtH[1]{u} + \wtH[1]{e}) \\
    % & \quad ( \wtL[2]{u - \ipoll u} + \wtH[1]{u - \ipoll u}
    %   + \wtL[2]{e} + \wtH[1]{e}) \\
    % & \leq ch^{2} ( \wtH[1]{u} + \wtH[1]{e})
    %   ( c h\abs{\log h}^{1/2} \Wnorm[\surface(t)][2,\infty]{u}
    %   + \wtL[2]{e} + \wtH[1]{e}) \\
    % & \leq c \wtH[1]{u}^{2} +
    %   c h^{6} \abs{\log h} \Wnorm[\surface(t)][2,\infty]{u}
    %   + c h^{4}\wtL[2]{e}^{2} + ch^{4} \wtH[1]{e}^{2} \\
    % & \quad ch^{4} \wtH[1]{e}^{2}
    %   + c h^{2} \abs{\log h} \Wnorm[\surface(t)][2,\infty]{u}
    %   + c \wtL[2]{e} + ch^{2} \wtH[1]{e}^{2} \\
    % & \overset{h < h_{1}}{\leq}
    %   c \wtH[1]{u}^{2} +
    %   c h^{6} \abs{\log h} \Wnorm[\surface(t)][2,\infty]{u}
    %   + c h^{4}\wtL[2]{e}^{2} + \frac{1}{3\cdot 8} \wtH[1]{e}^{2} \\
    % & \quad \frac{1}{3\cdot 8} \wtH[1]{e}^{2}
    %   + c h^{2} \abs{\log h} \Wnorm[\surface(t)][2,\infty]{u}
    %   + c \wtL[2]{e} + \frac{1}{3\cdot 8} \wtH[1]{e}^{2} \\
    % ------------------------------------------------------------
    & \leq \eps \wtH[1]{e}^{2} +
      c\bigl( h^{2} \abs{\log h} \Wnorm[\surface(t)][2,\infty]{u}
      + \wtL[2]{e}\bigr)
  \end{align*}
  These estimates together, and absorbing $\wtH[1]{e}^{2}$, imply \eqref{eq:46}.  \par
  \textbf{Step~2:} Using an Aubin--Nitsche argument we prove that there exists $\gamma >
  \gamma_{0}>0$ sufficiently large such
  that for all $\delta>0$ the following estimate holds
  \begin{equation}
    \label{eq:52}
    \wtL[2]{u - \Rmapl u}^{2}
    \leq c h^{4} \Wnorm[\surface(t)][2,\infty]{u}^{2} +
    \delta \wtH[1]{u - \Rmapl u}^{2}.
  \end{equation}
  %\eqref{eq:53} follows by using \eqref{eq:46} and \eqref{eq:52}.
  Let $w\in H^{2}(\surface(t))$ be the weak solution of
  \begin{equation*}
    -\lb w + w = \mu^{-2}e.
  \end{equation*}
  Then by testing with $e$ we obtain
  \begin{align*}
    \wtL[2]{e}^{2}
    & = \bigl(\aast(e,w) - \aast(e,\ipoll w) \bigr) + \aast(e,\ipoll w)
%    \\  &   = J_{1}  + J_{2}.
    = \aast( e, w - \ipoll w) + \aast(e,\ipoll w)
  \end{align*}
  In addition to the already mentioned lemmata in Step~1 use
  Lemma~\ref{lemma:wt_ellliptic}~\eqref{eq:24},
  Lemma~\ref{lemma:calculations_with_weight}~\eqref{eq:51},
  Lemma~\ref{lemma:wt_eigenvalues}~\eqref{eq:45} and a sufficiently large
  $\gamma> \gamma_{0}>0$ to estimate
  \begin{align*}
%    \abs{J_{1}}
    |\aast( e, w - \ipoll w)|
    % ------------------------------------------------------------
    % Details
% & = \aast( e, w - \ipoll w) \\
% & \leq \wtH[1]{e} \wtH[-1]{w - \ipoll w} \\
% & \leq c \wtH[1]{e}  h \wtHk[2][-1]{w} \\
% & \leq c \wtH[1]{e} h \bigl( \wtL[-1]{\mu^{-2}e}
%   + \Hnorm[\surface(t)]{w}\bigr) \\
% & \leq c \wtH[1]{e} h \bigl( \rho^{-1} \wtL[2]{e}
%   + \rho^{-1} \abs{\log \rho}^{1/2} \wtL[2]{e}\bigr) \\
% & \leq c h \rho^{-1}\abs{\log \rho}^{1/2} \wtH[1]{e} \wtL[2]{e} \\
% & \overset{\gamma > \gamma_{0}}{\leq} \sqrt{\frac{\delta}{2}}
%   \wtH[1]{e} \wtL[2]{e} \\
% & \leq \frac{1}{4} \wtL[2]{e}^{2}
%   + \frac{\delta}{2} \wtH[1]{e}^{2} \\
    % ------------------------------------------------------------
    & \leq \frac{1}{4}\wtL[2]{e}^{2} + \frac{\delta}{2} \wtH[1]{e}^{2}.
  \end{align*}
  For the other term we estimate
  \begin{align*}
%    \abs{J_{2}}
    |\aast(e,\ipoll w)|
    & \leq c h^{2} \Hnorm{e} \Hnorm{\ipoll w}
    % ------------------------------------------------------------
    % Details
    % \\
    % & \leq c h^{3} \Hnorm[][2]{u} ( \Hnorm{w - \ipoll w} + \Hnorm{w}) \\
    % & \leq c h^{3} \Hnorm[][2]{u}
    %   ( h \Hnorm[][2]{w} + \rho^{-1}\abs{\log \rho}^{1/2}
    %   \wtL[2]{e}) \\
    % & \leq c h^{2} \Hnorm[][2]{u} ( h^{2} \Lnorm{\mu^{-2}e} +
    %   h\rho^{-1}\abs{\log \rho}^{1/2} \wtL[2]{e}) \\
    % & \leq c h^{2} ( h^{2} \rho^{-2} \wtL[2]{e}
    %   + h\rho^{-1}\abs{\log \rho}^{1/2}
    %   \wtL[2]{e}) \\
    % & \leq c h^{2} \Hnorm[][2]{u} \wtL[2]{e} \\
    % & \leq c h^{4} \Hnorm[][2]{u}^{2}  + \frac{1}{4} \wtL[2]{e}^{2} \\
    % &
    % ------------------------------------------------------------
      \leq c h^{4} \Wnorm[\surface(t)][2,\infty]{u}^{2}
      + \frac{1}{4} \wtL[2]{e}^{2}.
  \end{align*}
  By absorption, this implies \eqref{eq:52}.

  The final estimate is shown by combining \eqref{eq:46} and \eqref{eq:52}, and choosing $\delta>0$ such that $\hat c \delta <1$. Then an absorbtion finishes the proof.
\end{proof}

\begin{theorem}
  \label{theorem: L infty errors in Ritz}
  There exist constants $c>0$ independent of $h$ and $t$ such that
  \begin{equation*}
  \begin{aligned}
    \Lnorm[\surface(t)][\infty]{u - \bigl(\Rmap(t) u\bigr)^{l}}
    \leq&\ c h^{2}\abs{\log h}^{3/2} \Wnorm[\surface(t)][2,\infty]{u}, \\
    \Wnorm[\surface(t)][1,\infty]{u - \bigl(\Rmap(t) u\bigr)^{l}}
    \leq&\ c h\abs{\log h} \Wnorm[\surface(t)][2,\infty]{u},
  \end{aligned}
  \quad (u \in W^{2,\infty}(\Gat)).
  \end{equation*}
\end{theorem}

\begin{proof}
  Using Lemma~\ref{lemma: interpolation error},
  Lemma~\ref{lemma:w1i_to_wt_H1}~\eqref{eq:43} and Lemma~\ref{lemma:wtL2_to_Li_or_wtH1_to_W1i}~\eqref{eq:36} we get
  \begin{align*}
    \Wnorm[\surface(t)][1,\infty]{u - \Rmapl u}
    & \leq \Wnorm[\surface(t)][1,\infty]{u - \ipoll u}
      + c\Wnorm[\surface_{h}(t)][1,\infty]{\ipol u - \Rmap u} \\
    & \leq c h \Wnorm[\surface(t)][2,\infty]{u}
      + c \abs{\log h}^{1/2} \wtH[1]{\ipol u - \Rmap u} \\
    & \leq c h \abs{\log h} \Wnorm[\surface_{h}(t)][2,\infty]{u}
      + c \wtH[1]{u - (\Rmap u)^{l}}.
  \end{align*}
  For the $W^{1,\infty}$-estimate use Lemma~\ref{lemma:wt_ritz_error_estimate} to estimate the weighted norms. The $L^{\infty}$-estimate is obtained in a similar way.
\end{proof}

%\rev{Theorem 4.1 is basically already contained in [Demlow (2009)], Corollary 4.6. The main difference is that the constants in [Demlow (2009)] are not shown to be uniform in $t$. The proofs in [Demlow (2009)] rely on Aubin's surface Green's function estimates. Because the authors already show in Appendix A.\ (Theorem~A.1) that certain of these Green's function estimates have $t$-independent constants, it is likely that it could easily be shown that the bounds in [Demlow (2009)] do also.}{
\begin{remark}
The paper of Demlow \cite{Demlow2009} (dealing with elliptic problems on stationary surfaces) contains a related result in Corollary~4.6, however it does not directly imply Lemma~\ref{lemma:wt_ritz_error_estimate}. There are two crucial differences compered to the theorem above. Since there is no surface evolution in \cite{Demlow2009} the constants appearing in his proof would need to be shown being uniform in time\footnote{In fact some of them is later shown to be $t$-independent in the appendix.}. Furthermore, Demlow uses a different Ritz map (denoted by $\tilde u_{hk}^\ell$ there): instead of using the positive definite bilinear form $a^*(\cdot,\cdot)$ in \eqref{eq:39}, he uses the original positive \emph{semi-definite} bilinear form $a(\cdot,\cdot)$ and works with functions with mean value zero.
\end{remark}
%}

% \subsection{Discrete material derivative error estimates for a Ritz map}
\subsection{Maximum norm material derivative error estimates}
%\label{sec:discr-mater-deriv}

Since the material derivative does not commute with the \emph{time dependent} Ritz map, i.e.\ $\mat_h \lRmap(t) u \neq \lRmap(t) \mat_{h} u$, we bound the error $\mat_{h}(u - \lRmap(t) u)$. Again we first show our estimates in the weighted norms, and then use these results for the $L^\infty$- and $W^{1,\infty}$-norm error estimates. Up to the authors knowledge such a maximum norm error estimate for the material derivative of the Ritz map have not been shown in the literature before.

For this subsection we write $\Rmap u$ instead of $\Rmap(t)u$ and further $\lRmap u$ instead of $\lRmap(t) u$.

We first state a substitute for our weighted pseudo Galerkin inequality \eqref{eq:49}.
\begin{lemma}
  \label{lemma:wt_md_pseudo_galerkin_eq}
  There exists a constant $c>0$ independent of $h$ and $t$ such that for all
  $u\in W^{2,\infty}(\GT)$ and $\varphi_{h}\in S_{h}^{l}(t)$ it holds
  \begin{equation}
    \label{eq:59}
    \begin{split}
      \abs{\aast( \mat_{h} (u & - \Rmapl u), \varphi_{h} )} \leq c
      \Big(h^{2} \wtH[1]{\mat_h (u - \Rmapl u)} \\
      & \quad + h \abs{\log h}^{1/2} \big(\Wnorm[\surface(t)][2,\infty]{u} +
      \Wnorm[\surface(t)][1,\infty]{\mat u}\big)\Big) \wtH[-1]{\varphi_{h}}.
    \end{split}
  \end{equation}
\end{lemma}

\begin{proof}
  The main idea is given by Dziuk and Elliott in \cite{DziukElliott_L2}.  Using
  \eqref{eq:40} and Lemma~\ref{lemma:wtL2_to_Li_or_wtH1_to_W1i}~\eqref{eq:36}
  it is easy to verify
  \begin{align}
    \wtH[1]{\mat_{h}\Rmapl u} \leq &\ \wtH[1]{\mat_{h} u - \mat_{h}\Rmapl u} \nonumber \\
    &\ + c\abs{\log h}^{1/2} \big( \Wnorm[\surface(t)][1,\infty]{\mat u}
    + h \Wnorm[\surface(t)][2,\infty]{u} \big). \label{eq:63}
  \end{align}
  Let $\phi_{h}\in S_{h}(t)$ such that $\varphi_{h}=\phi_{h}^{l}$.  Taking time
  derivative of the definition of the Ritz map \eqref{eq:39}, using the
  discrete transport properties \eqref{eq: continuous transport properties} Lemma~\ref{lemma: transport properties}, and the definition of the Ritz map, we obtain
  \begin{equation}
    \label{eq:65}
    \begin{aligned}
      \aast(\mat_{h} u - \mat_{h} \Rmapl u, \varphi_{h}) =&\
      \aast_{h}(\mat_{h}\Rmap u, \phi_{h}) - \aast(\mat_{h}\Rmapl u, \varphi_{h})  \\
      &\ + (g_{h}+b_{h})(V_{h}; u^{-l},\phi_{h}) - (g+b)(v_{h}; u, \varphi_{h}) \\
      &\ -  (g_{h}+b_{h})(V_{h}; u^{-l} - \Rmap u, \phi_{h}) .
    \end{aligned}
  \end{equation}
  Then estimate using Lemma~\ref{lemma:properties_of_wt_norm}~\eqref{eq:61},
  \eqref{eq:62},
  Lemma~\ref{lemma:wt_ritz_error_estimate}~\eqref{eq:53} and the above
  inequality to finish the proof (cf.\ \cite[Theorem 7.2]{diss_Mansour}).
\end{proof}

\begin{lemma}
  \label{lemma:md_ipol_err}
  For $k\in \{0,1\}$ there exists $c=c(k)>0$ independent of $t$ and $h$ such
  that for $u\in W^{3,\infty}(\GT)$ the following inequalities hold
    \begin{align}
    \label{eq:35}
        \Wnorm[\surface(t)][k,\infty]{\mat_{h} u - \ipoll \mat u}
        \leq &\ c h^{2-k} \big(\Wnorm[\surface(t)][2,\infty]{u}
        + \Wnorm[\surface(t)][2,\infty]{\mat u}\big), \\
        \wtL[2]{\mat_{h} u - \ipoll \mat u}^{2} +&\ \wtH[1]{\mat_{h}u - \ipoll \mat u}^{2} \nonumber \\
        \leq &\ c h^{2}\abs{\log h} \big(\Wnorm[\surface(t)][2,\infty]{u} + \Wnorm[\surface(t)][2,\infty]{\mat u}\big).\label{eq:58}
    \end{align}
\end{lemma}

\begin{proof}
  Using \eqref{eq:40} we get
  \begin{multline*}
    \Wnorm[\surface(t)][k,\infty]{\mat_{h} u - \ipoll \mat u} \\
    \leq \Wnorm[\surface(t)][k,\infty]{(v-v_{h}) \cdot \nbg u}
    + \Wnorm[\surface(t)][k,\infty]{\mat u - \ipoll \mat u}.
  \end{multline*}
  Use Lemma~\ref{lemma: interpolation error} and \eqref{eq: velocity L infty
    bound} to show the first estimate.%\eqref{eq:35}.

  For the second inequality %\eqref{eq:58}
  use a H\"{o}lder estimate, and \eqref{eq:35} with
  Lemma~\ref{lemma:wtL2_to_Li_or_wtH1_to_W1i} \eqref{eq:33} and \eqref{eq:36}.
\end{proof}

\begin{lemma}
  \label{lemma:wt_md_ritz_error_estimate}
  There exists $h_{0}>0$ sufficiently small and $\gamma_{0}>0$ sufficiently
  large and a constant $c=c(h_{0},\gamma_{0})>0$ such that for $u\in
  W^{3,\infty}(\GT)$ the following holds
  \begin{equation}
    \label{eq:55}
    \begin{split}
      \wtL[2]{\mat_{h}u-\mat_{h}\lRmap u}^{2}
      &+ \wtH[1]{\mat_{h}u-\mat_{h}\lRmap u}^{2} \\
      & \leq c h^{2} \abs{\log h}^{4} (\Wnorm[\surface(t)][2,\infty]{u}^{2}
      + \Wnorm[\surface(t)][2,\infty]{\mat u}^{2}).
    \end{split}
  \end{equation}
\end{lemma}

\begin{proof}
  This proof has a similar structure as Lemma~\ref{lemma:wt_ritz_error_estimate},
  and since it also uses similar arguments, we only give references if new lemmata are
  needed.  For the ease of presentation we set $e= u - \Rmapl u$ and split the error as follows
  \begin{equation*}
    \mat_{h}e =
    ( \mat_{h}u - \ipoll \mat u ) + (\ipoll \mat u - \mat_{h}\Rmapl u)
    =: \sigma + \theta_{h}.
  \end{equation*}
  \textbf{Step~1:}  Our goal is to prove   \begin{equation}
    \label{eq:32}
    \wtH[1]{ \mat_{h}e}^{2} \leq ch^{2} \abs{\log h} \big(\Wnorm[\surface(t)][2,\infty]{ u}^{2} +
    \Wnorm[\surface(t)][2,\infty]{\mat u}^{2} \big) + \hat c \wtL[2]{\mat_{h}e}^{2}.
  \end{equation}
  We start with
  \begin{equation*}
    \wtH[1]{\mat_{h}e}^{2} \leq a^{*}\Bigl(\mat_{h}e,
    \mu^{-1}\mat_{h}e \Bigr) + c \wtL[2]{\mat_{h} e}^{2}
  \end{equation*}
  and continue with
  \begin{align*}
    \aast(\mat_{h}e, \mu^{-1}\mat_{h}e)
    & = \aast(\mat_{h}e, \mu^{-1} \sigma) \\
    & \quad + \aast\bigl( \mat_{h}e, \mu^{-1}\theta_{h}
      - I(\mu^{-1}\theta_{h})\bigr) \\
    & \quad + \aast\bigl(\mat_{h}e, I(\mu^{-1}\theta_{h})\bigr)
%    \\ &
    = I_{1} + I_{2} + I_{3}.
  \end{align*}
  We estimate the three terms separately. For the first $\varepsilon$-Young inequality and Lemma~\ref{lemma:md_ipol_err}~\eqref{eq:58} yields
  \begin{align*}
    \abs{I_{1}}
    % ------------------------------------------------------------
    % Details
    % & \leq \wtH[1]{\mat_{h}e} \wtH[-1]{\mu^{-1}\sigma} \\
    % & \leq c \wtH[1]{\mat_{h}e} ( \wtL[2]{\sigma}+ \wtH[1]{\sigma}) \\
    % & \leq c \wtH[1]{\mat_{h}e} \frac{\rho}{\sqrt{\gamma}}
    %   ( \Wnorm[][2,\infty]{u} + \Wnorm[][2,\infty]{\mat u}) \\
    % ------------------------------------------------------------
    & \leq \varepsilon \wtH[1]{\mat_{h}e}^{2}
      + c h^{2}\abs{\log h}
      (\Wnorm[\surface(t)][2,\infty]{u}^{2}
      + \Wnorm[\surface(t)][2,\infty]{\mat u}^{2}).
  \end{align*}
  For a sufficiently small $0<h< h_{0}$ we obtain
  \begin{align*}
    \abs{I_{2}}
    % ------------------------------------------------------------
    % Details
    % & \leq \wtH[1]{\mat_{h}e}
    %   \wtH[-1]{\mu^{-1}\theta_{h}- I(\mu^{-1}\theta_{h})} \\
    % & \overset{h<h_{0}}{\leq}
    %   \varepsilon \wtH[1]{\mat_{h}e}
    %   (\wtL[2]{\theta_{h}}+ \wtH[1]{\theta_{h}}) \\
    % & \leq \varepsilon \wtH[1]{\mat_{h}e}
    %   ( \wtL[2]{\mat_{h}e} + \wtH[1]{\mat_{h}e}
    %   + \wtL[2]{\sigma}+ \wtH[1]{\sigma}) \\
    % & \leq \varepsilon \wtH[1]{\mat_{h}e}
    %   \bigl( \wtL[2]{\mat_{h}e} + \wtH[1]{\mat_{h}e}
    %   + \frac{\rho}{\sqrt{\gamma}}
    %    ( \Wnorm[][2,\infty]{u}+ \Wnorm[]{2,\infty}{\mat u}) \bigr) \\
    % ------------------------------------------------------------
    & \leq \varepsilon \wtH[1]{\mat_{h}e}^{2}
      + c\bigl( \wtL[2]{\mat_{h}e}^{2}
      + h^{2}\abs{\log h} (\Wnorm[\surface(t)][2,\infty]{u}^{2}
      + \Wnorm[\surface(t)][2,\infty]{\mat u}^{2})\bigr).
  \end{align*}
  Using Lemma~\ref{lemma:wt_md_pseudo_galerkin_eq}~\eqref{eq:59} and a
  $0<h<h_{1}$ sufficiently small we arrive at
  \begin{align*}
    \abs{I_{3}}
    % ------------------------------------------------------------
    % Details
  %   & \leq c\bigl( h^{2}\wtH[1]{\mat_{h}e}
  %     + \frac{\rho}{\sqrt{\gamma}} (\Wnorm[][2,\infty]{u}
  %     + \Wnorm[][2,\infty]{\mat u})\bigr)
  %     \wtH[-1]{\ipoll (\mu^{-1}\theta_{h})}.
  % \end{align*}
  % Now observe
  % \begin{align*}
  %   \wtH[-1]{\ipoll (\mu^{-1}\theta_{h})}
  %   & \leq \wtH[-1]{\ipoll (\mu^{-1}\theta_{h})- \mu^{-1}\theta_{h}}
  %     + \wtH[-1]{\mu^{-1}\theta_{h}} \\
  %   & \leq c ( \wtL[2]{\theta_{h}}+ \wtH[1]{\theta_{h}}) \\
  %   & \leq c \bigl(\wtL[2]{\mat_{h}e} + \wtH[1]{\mat_{h}e}
  %     + \frac{\rho}{\sqrt{\gamma}}(\Wnorm[][2,\infty]{u}
  %     + \Wnorm[][2,\infty]{\mat u})\bigr).
  % \end{align*}
  % For a sufficiently small $0<h<h_{1}$ we reach
  % \begin{align*}
  %   \abs{I_{3}}
    % ------------------------------------------------------------
    & \leq \varepsilon \wtH[1]{\mat_{h}e}^{2}
      + c\bigl( \wtL[2]{\mat_{h}e}^{2} + h^{2}\abs{\log h}
      (\Wnorm[\surface(t)][2,\infty]{u}^{2}
      + \Wnorm[\surface(t)][2,\infty]{\mat u}^{2})
      \bigr).
  \end{align*}
  These estimates together, and absorbing $\wtH[1]{\mat_{h}e}$, imply \eqref{eq:32}.  \par
  \textbf{Step~2:} Using again an Aubin--Nitsche like argument we show that, for any $\delta>0$ sufficiently small, we have
  \begin{equation}
    \label{eq:64}
    \wtL[2]{ \mat_{h}e}^{2} \leq \delta \wtH[1]{\mat_{h}e}^{2}
    + c h^{2}
    \abs{\log h}^{4}\big( \Wnorm[\surface(t)][2,\infty]{ u}^{2} +
    \Wnorm[\surface(t)][2,\infty]{\mat u}^{2} \big)  .
  \end{equation}
  Let $w\in H^{2}\bigl(\surface(t)\bigr)$ be the weak solution of
  \begin{equation*}
    -\lb w + w = \mu^{-2} \mat_{h}e.
  \end{equation*}
  Then we have
  \begin{equation*}
    \wtL[2]{\mat_{h}e}
    = \aast(\mat_{h}e, w - \ipoll w) +  \aast(\mat_{h}e, \ipoll w) .
%    = I_{1}+I_{2}.
  \end{equation*}
  Again let $\varepsilon>0$ be a small number.  For $\gamma > \gamma_{0}$ sufficiently big we get
  \begin{align*}
    \abs{\aast(\mat_{h}e, w - \ipoll w)} \leq {}
    % ------------------------------------------------------------
    % Details
% & \wtH[1]{\mat_{h} e} \wtH[-1]{w - \ipoll w} \\
% \leq {}
% & c\wtH[1]{\mat_{h}e} h \wtHk[2][-1]{w} \\
% \leq {}
% & c \wtH[1]{\mat_{h}e}h ( \wtL[-1]{\mu^{-2} \mat_{h}e}
%   + \Hnorm{w}) \\
% \leq {}
% & c \wtH[1]{\mat_{h}e} h ( \rho^{-1}
%   + \rho^{-1}\abs{\log \rho}^{1/2} ) \wtL[2]{\mat_{h}e} \\
% \overset{\gamma > \gamma_{0}}{\leq} {}
% & \sqrt{\delta} \sqrt{\varepsilon}
%   \wtH[1]{\mat_{h}e} \wtL[2]{\mat_{h}e} \\
% \leq {}
% & \delta \wtH[1]{\mat_{h}e}^{2}
%   + \varepsilon \wtL[2]{\mat_{h}e}^{2} \\
% \leq {}
    % ------------------------------------------------------------
    & \varepsilon \wtH[1]{\mat_{h}e}^{2}
      + \delta \wtL[2]{\mat_{h}e}^{2}
  \end{align*}
  Using equation~\eqref{eq:65} and proceeding similar like in Dziuk and Elliott
  \cite[Theorem~6.2]{DziukElliott_L2}, by adding and subtracting terms, we get
  \begin{align*}
    \aast(\mat_{h}e, \ipoll w) = {}
    & - \Big( \aast(\mat_{h}\Rmapl u, \ipoll w)
      - \aast_{h}(\mat_{h}\Rmap u, \ipol w) \\
    & + (g+b)(v_{h};u, \ipoll w)
      - (g_{h}+b_{h})(V_{h}; u^{-l}, \ipol w) \\
    & + (g_{h}+b_{h})(V_{h}; u^{-l}- \Rmap u , \ipol w)
      - (g+b)(v_{h};u - \Rmapl u, \ipoll w) \\
    & + (g+b)(v_{h}; u-\Rmapl u, \ipoll w) - (g+b)(v; u-\Rmapl u,\ipoll w) \\
    & + (g+b)(v;u-\Rmapl u, \ipoll w) - (g+b)(v;u-\Rmapl u, w) \\
    & + (g+b)(v;u-\Rmapl u, w)
      \Big) \\
    = {}
    & J_{1}+J_{2}+J_{3}+J_{4}
    +J_{5}+J_{6}.
  \end{align*}
  Use Lemma~\ref{lemma:properties_of_wt_norm}~\eqref{eq:62}, \eqref{eq:63},
  Lemma~\ref{lemma:wt_ritz_error_estimate}~\eqref{eq:53} and the inequality
  \begin{equation*}
    h\wtH[1]{\ipoll w} \leq
    \varepsilon \wtL[2]{\mat_{h}e},
  \end{equation*}
  for $\gamma>\gamma_{1}$ sufficiently big, we reach at
  \begin{multline*}
    \abs{J_{1}}+ \dotsb %\abs{J_{2}} + \abs{J_{3}}
    + \abs{J_{4}} \leq \delta \wtH[1]{\mat_{h}e}^{2}
      + \varepsilon \wtL[2]{\mat_{h}e}^{2}
      + ch^{2}(\Wnorm[][2,\infty]{u}^{2}
      + \Wnorm[][1,\infty]{\mat u}^{2}).
  \end{multline*}
  With the same arguments like for $\aast(\mat_{h}e, w - \ipoll w)$ we estimate
  \begin{equation*}
    \abs{J_{5}} \leq
    \varepsilon \wtH[1]{\mat_{h}e}^{2} + \delta \wtL[2]{\mat_{h}e}^{2},
  \end{equation*}
  for $\gamma>\gamma_{2}$ sufficiently big.  For \(\gamma> \gamma_{3}\)
  sufficiently big we estimate the last term as
  follows
  \begin{align*}
    \abs{J_{6}} \leq {}
    & c \wtL[1]{e} \wtHk[2][-1]{w} \\
    \leq {}
    & c \Lnorm[][\infty]{e} \abs{\log \rho}^{1/2}
      \wtHk[2][-1]{w} \\
    \leq {}
    & c h^{2} \abs{\log h}^{3/2} \Wnorm[][2,\infty]{u}
      \wtHk[2][-1]{w} \\
    \leq {}
    & \varepsilon \wtL[2]{\mat_{h}e}^{2}
      + ch^{2} \abs{\log h}^{4} \Wnorm[][2,\infty]{u}^{2}.
  \end{align*}
  By absorption, these estimates together imply \eqref{eq:64}.

  The final estimate is shown by combining \eqref{eq:32} and \eqref{eq:64}, and choosing $\delta>0$ such that $\hat c \delta <1$. Then an absorbtion finishes the proof.
\end{proof}

From the weighted version of the error estimate in the material derivatives, the $L^\infty$-norm estimate follows easily.
\begin{theorem}[Errors in the material derivative of the Ritz projection]
  \label{theorem: L infty errors in mat derivative of Ritz}
  Let $z\in W^{3,\infty}(\GT)$.  For a sufficiently small $h < h_{0}$ and a
  sufficiently big $\gamma> \gamma_{0}$ there exists $c= c(h_{0},\gamma_{0})>0$
  independent of $t$ and $h$ such that
  \begin{align*}
    \Lnorm[\surface(t)][\infty]{\mat_{h} \bigl(z
    - (\Rmap(t)z)^{l}\bigr)} %\\
    &\leq ch^{2}\abs{\log h}^{3} \bigl(\Wnorm[\surface(t)][2,\infty]{z} +
    \Wnorm[\surface(t)][2,\infty]{\mat z}\bigr), \\
    \Wnorm[\surface(t)][1,\infty]{\mat_{h} \bigl(z
    -(\Rmap(t)z\bigr)^{l})}% \\
    & \leq ch \abs{\log h}^{5/2} \bigl(\Wnorm[\surface(t)][2,\infty]{z} +
    \Wnorm[\surface(t)][2,\infty]{\mat z}\bigr).
  \end{align*}
\end{theorem}
\begin{proof}
The above results are shown by exactly following the proof of
Theorem~\ref{theorem: L infty errors in Ritz},
Lemma~\ref{lemma:wt_md_ritz_error_estimate}~\eqref{eq:55} being the main tool.
\end{proof}

\section{Maximum norm parabolic stability}
\label{sec:weak-discr-maxim-1}

The purpose of this section is to derive a ESFEM weak discrete maximum
principle.  The proof is modeled on the weak discrete maximum principle from
Schatz, Thom\'{e}e, Wahlbin \cite{1980_stw}.  For
this we are going to need a well known matrix formulation of
\eqref{eq: semidiscrete problem}, which is due to Dziuk and Elliott
\cite{DziukElliott_ESFEM}.  It
was first used in Dziuk, Lubich, Mansour \cite{DziukLubichMansour_rksurf} for
theoretical reasons, namely a
time discretization of \eqref{eq: semidiscrete problem}.  Using the matrix
formulation we derive a discrete adjoint problem of \eqref{eq: semidiscrete
  problem}, which does not arise in Schatz, Thom\'{e}, Wahlbin \cite{1980_stw},
but arises here, since the ESFEM evolution operator is not
self adjoint.  Then we deduce a corresponding a priori estimate and finally
prove our weak discrete maximum principle.

\subsection{A discrete adjoint problem}

A matrix ODE version of \eqref{eq: semidiscrete problem} can be derived by setting
\begin{equation*}
    U_h( \, . \,,t) = \sum_{j=1}^N \alpha_j\t \chi_j( \, . \,,t),
\end{equation*}
testing with the basis function $\phi_h=\chi_j$, where \(S_{h}(t) = \linspan \{
\chi_{j}\mid j =1,\dotsc, N\} \), and using the transport property
\eqref{eq: transport property}.

\begin{proposition}[ODE system]
%\label{prop: ODE system}
    The spatially semidiscrete problem \eqref{eq: semidiscrete problem} is equivalent to the following linear ODE system for the vector $\alpha\t=(\alpha_j\t)\in\R^N$, collecting the nodal values of $U_h(.,t)$:
    \begin{equation}\label{eq: ODE system}
        \disp
        \begin{cases}
            \begin{alignedat}{2}
                \diff \Big(M\t \alpha\t\Big) + A\t \alpha\t &= 0 \\
                \alpha(0) &= \alpha_0
            \end{alignedat}
        \end{cases}
    \end{equation}
    where the evolving mass matrix $M\t$ and stiffness matrix $A\t$ are defined as
    \begin{equation*}
        M(t)_{kj} = \int_{\Ga_h\t}\!\!\!\! \chi_j \chi_k, \qquad A\t_{kj} =
        \int_{\Ga_h\t}\!\!\! \nbgh \chi_j \cdot \nbgh  \chi_k.
    \end{equation*}
\end{proposition}

\begin{definition}
  Let $0\leq s \leq t \leq T$.  For given initial value $w_{h}\in S_{h}(s)$ at time $s$,
  there exists unique\footnote{cf.\ Dziuk and Elliott
    \cite{DziukElliott_ESFEM}. } solution $u_{h}$.  This defines a linear
  evolution operator
  \begin{equation*}
    E_{h}(t,s)\colon S_{h}(s)\to S_{h}(t),\quad w_{h}\mapsto u_{h}(t).
  \end{equation*}
  We define the adjoint of $E_{h}(t,s)$
  \begin{equation*}
    E_{h}(t,s)^{*}\colon S_{h}(t) \to S_{h}(s)
  \end{equation*}
  via the equation
  \begin{equation}
    \label{eq:21}
    m_{h}\bigl(t; E_{h}(t,s)\varphi_{h}(s), w_{h}(t)\bigr) = m_{h}\bigl(s;
    \varphi_{h}(s), E_{h}(t,s)^{*} w_{h}(t)\bigr),
  \end{equation}
  where $\varphi_{h}(s)\in S_{h}(s)$ and $w_{h}(t) \in S_{h}(t)$ are some arbitrary
  finite element functions.
\end{definition}

\begin{lemma}[Adjoint problem]
  \label{lemma:adjoint_problem}
  Let $s\in [0,t]$ where $t\in[0,T]$ and $w_{h}(t)\in S_{h}(t)$.  Then $s\mapsto
  E(t,s)^{*}w_{h}(t)$ is the unique solution of
  \begin{align}
    \label{eq:22}
    \left\{
    \begin{aligned}
    m_{h}\bigl(s;\partial^{\bullet,s}_{h} u_{h}, \varphi_{h}\bigr) - a_{h}(s;
    u_{h},\varphi_{h}) =&\ 0,  &\quad& \textrm{ on } \Ga(s)\\
    u_{h}(t) =&\ w_{h}(t),  &\quad& \textrm{ on } \Ga(t).
    \end{aligned}
    \right.
  \end{align}
  where $\partial^{\bullet,s}_{h}$ is the discrete material derivative with respect to
  $s$.
\end{lemma}

\begin{remark}
  The problem \eqref{eq:22} has the structure of a backward heat equation, where $s$ is
  going backward in time.  Hence we considered \eqref{eq:22} as a PDE of
  parabolic type.  We recall, that using Lemma~\ref{lemma: transport properties} we may write
  equation~\eqref{eq: semidiscrete problem} equivalently as
  \begin{equation}
    \label{eq:20}
    \left\{
    \begin{aligned}
      m_{h}\bigl(t;\dellmat_{h} u_{h} + (\nbgh\cdot V_{h}) u_{h},\varphi_{h}\bigr)
      + a_{h}(t; u_{h},\varphi_{h}) &= 0, &\quad& \textrm{ on } \Ga(t),\\
      u_{h}(0) &= w_{h}, &\quad& \textrm{ on } \Ga(0)
    \end{aligned}
    \right.
  \end{equation}
  The problems \eqref{eq:20} and \eqref{eq:22} differ in the following way: If the
  initial data for \eqref{eq:22} is constant then it remains so for all times.
  In general this does not hold for solutions of \eqref{eq:20}.  On the other
  hand \eqref{eq:20} preserves the mean value of its initial data, which is in
  general not true for a solution of \eqref{eq:22}.
\end{remark}

\begin{proof}[Proof of Lemma~\ref{lemma:adjoint_problem}.]
  First we investigate the finite element matrix representation of $E_{h}(t,s)$
  with respect to the standard finite element basis, which we denote by
  $\bm{E}_{h}(t,s)$.  From \eqref{eq: ODE system} we have
  \begin{equation*}
  \ddti{}{t}\bigl(M(t) \bm{E}_{h}(t,0)\bm{u}_{h}(0)\bigr) + A(t)
  \bm{E}_{h}(t,0)\bm{u}_{h}(0) = 0.
  \end{equation*}
  Let $\Lambda(t,s)$ the resolvent matrix of the ODE
  \begin{equation*}
    \ddti{\xi}{t} + A(t)M(t)^{-1} \xi = 0.
  \end{equation*}
  Then obviously it holds
  \begin{equation*}
    \bm{E}_{h}(t,s)  = M(t)^{-1} \Lambda(t,s) M(s).
  \end{equation*}
  Denote by $\bm{E}_{h}(t,s)^{*}$ the matrix representation of $E_{h}(t,s)^{*}$.
  From equation~\eqref{eq:21} it follows
  \begin{equation*}
    \bm{E}_{h}(t,s)^{*} = M(s)^{-1} \bm{E}_{h}(t,s)^{T} M(t) = \Lambda(t,s)^{T}.
  \end{equation*}
  Now we calculate $\ddti{\Lambda(t,s)}{s}$.  Note that $\Lambda(t,s)=
  \Lambda(s,t)^{-1}$ and it holds
  \begin{equation*}
    \ddti{\Lambda(s,t)^{-1}}{s} = -\Lambda(s,t)^{-1}\ddti{\Lambda(s,t)}{s}
    \Lambda(s,t)^{-1}.
  \end{equation*}
  From that it easily follows
  \begin{equation*}
    \ddti{\Lambda(t,s)}{s} = \Lambda(t,s) A(s)M(s)^{-1},
  \end{equation*}
  which now implies
  \begin{equation*}
    \ddti{\bm{E}_{h}(t,s)^{*}}{s} = M(s)^{-1}A(s)\bm{E}_{h}(t,s)^{*}. \qedhere
  \end{equation*}
\end{proof}

\subsection{A discrete delta and Green's function}

Let $\delta_{h} = \delta_h^{x_h} = \delta_{h}^{t,x_{h}} \in S_h\t$ be a finite
element discrete delta function defined as
\begin{equation}
  \label{eq:15}
  m_h(t;\delta_h^{t,x_h},\vphi_h) = \vphi_h(x_h,t) \qquad (\vphi_h \in S_h\t).
\end{equation}
If \(\delta^{x_{h}}\colon \surface_{h}(t)\to \R\)
is a smooth function having support in the triangle \(E_{h}\)
containing \(x_{h}\),
then since \(\dim \surface_{h}(t) = 2 \)
one easily calculates
\(\Lnorm[\surface_{h}(t)]{\delta^{x_{h}} \pwt^{x_{h}}} \leq c\)
for some constant independent of \(h\)
and \(t\).  For the discrete delta function \(\delta_{h}\) a similar
result holds.

\begin{lemma}
  \label{lemma:delta_rho_bound}
  There exists $c>0$ independent of $t$ and \(h\):
  \begin{equation*}
    \|\pwt^{x_{h}} \delta_h^{x_h} \|_{L^2(\Ga_h\t)}
    \leq c \qquad (x_h \in \Ga_h\t).
  \end{equation*}
\end{lemma}

The proof is a straight forward extension of the corresponding one in Schatz,
Thom\'{e}e, Wahlbin \cite{1980_stw} and uses the exponential decay property of
the \(L^{2}\)-projection, cf. Theorem~\ref{theorem:old_L2_bound}~\eqref{eq:14}.
% \begin{proof}
% We again fix the time $t$, and $x_h\in\Ga_h\t$, and then drop the indices. We define
% \begin{align*}
%     \Ga_j\t =&\ \big\{ y_h\in\Ga_h\t \,\,\big|\,\, 2^{j-1}h < |x_h-y_h| \leq 2^{j}h \big\}, \\
%     \Ga_0\t =&\ \big\{ y_h\in\Ga_h\t \,\,\big|\,\, |x_h-y_h| \leq h \big\}.
% \end{align*}
% Using \eqref{eq:42}, % Lemma~\ref{lemma: rho equivalence},
% we obtain the bound
% \begin{equation*}
%     \pwt^{-l}(y_h) \leq c h 2^j \for y_h \in \Ga_j\t,
% \end{equation*}
% and hence we also have
% \begin{equation}
% \label{eq: exponentially scaled rho-delta est}
%     \|\pwt^{-l} \delta\| \leq c \sum_{j\geq0} h 2^j \|\delta\|_{L^2(\Ga_j\t)}
% \end{equation}
% The norm of $\delta$ can be using
% \begin{equation*}
%     \|\delta\|_{L^2(\Ga_j\t)} = \sup_{\substack{\vphi\in C_0^\infty(\Ga_j\t), \\ \|\vphi\|_{L^2(\Ga_j\t)}=1}} m_h(\delta,\vphi)\ .
% \end{equation*}
% Let us temporarily fix $\vphi$, and consider the triangle $E_h$ containing $x_h$. By inverse inequality, the exponential decay property of $\Lproj$, and the trivial estimate $|x_h-y_h| \leq \dist_h(x_h,y_h)$, we obtain
% \begin{align*}
%     m_h(\delta,\vphi) =&\ m_h(\delta,\Lproj \vphi) = (\Lproj \vphi)(x_h) \\
%     \leq &\ c h\inv \|\Lproj \vphi\|_{L^2(E_h)} \leq c h\inv e^{-c2^j} \|\vphi\|_{L^2(\Ga_j)}.
% \end{align*}
% Therefore, by reinserting the estimate
% \begin{equation*}
%     \|\delta\|_{L^2(\Ga_j)} \leq c h\inv e^{-c2^j}
% \end{equation*}
% into \eqref{eq: exponentially scaled rho-delta est}, we completed the proof.
% \end{proof}
\bigskip

Next we define a finite element discrete Green's function as follows.  Let
$s\in [0,T]$.  For given $u_{h}\in S_{h}(s)$ there exists a unique
$\psi_{h}\in S_{h}(s)$ such that
\begin{equation*}
  a_{h}^{*}(s; \psi_{h}, \varphi_{h}) =
  m_{h}(s; u_{h}, \varphi_{h}) \quad \forall \varphi_{h}\in S_{h}(s).
\end{equation*}
This defines an operator
\begin{equation*}
  T_{h}^{*,s}\colon S_{h}(s)\to S_{h}(s),\quad T_{h}^{*,s}u_{h} := \psi_{h}.
\end{equation*}
We call $G^{s,x}_{h}:= T_{h}^{*,s}\delta_{h}^{s,x}$ a discrete Green's
function.  \par
A short calculation shows that for all $0\neq \varphi_{h}\in S_{h}(s)$ it holds
\begin{equation*}
  m_{h}(s;T_{h}^{*,s} \varphi_{h}, \varphi_{h}) >0,
\end{equation*}
which implies that $G_{h}^{s,x}(x) > 0$.  Actually we can bound the singularity
\(x\) with \(c \abs{\log h}\).

\begin{lemma}
  \label{lemma:discrete_greens_function}
  For the discrete Green's function $G^{s,x}_{h}$ we have the estimate
  \[
  G^{s,x}_{h}(x) \leq c \abs{\log h}.
  \]
\end{lemma}

\begin{proof}
  Using Lemma~\ref{lemma:improved_inverse_estimate} with \eqref{eq:66} we
  estimate as
  \begin{align*}
    \Lnorm[\surface_{h}(s)][\infty]{G^{s,x}_{h}}
    & \leq c \abs{\log h}^{1/2}
      \Hnorm[\surface_{h}(s)]{G^{s,x}_{h}}  %\\ &
    = c \abs{\log h}^{1/2} \sqrt{G^{s,x}_{h}(x)}.
  \end{align*}
\end{proof}

The next lemma needs a different treatment then the one presented in Schatz,
Thom\'{e}e and Wahlbin \cite{1980_stw}.  The reason for that is that the mass
and stiffness matrix depend on time and further the stiffness matrix is
singular.
% We cope the first problem with
% Lemma~\ref{lemma:discrete_transport_prperty}~\eqref{eq:60} and for the second
% problem we consider $(A+M)^{-1}$ instead of $A^{-1}$.

\begin{lemma}
  \label{lemma:adjoint_problem_estimate}
  Let be $u_{h}$ a solution of \eqref{eq:22}.  Then we have the estimate
  \[
  \int_{0}^{t}\Lnorm[\surface_{h}(s)]{u_{h}}^{2} \mathrm{d}s \leq c \cdot
  m_{h}(t; T_{h}^{*,t}u_{h}, u_{h}).
  \]
\end{lemma}

\begin{proof}
  Note that Lemma~\ref{lemma: transport properties}~\eqref{eq: discrete transport properties} reads
  with the matrix notation as follows: If $\bm{Z}_{h}$ and $\bm{\phi}_{h}$ are
  the coefficient vectors of some finite element function, then we have the
  estimate
  \begin{equation}
    \label{eq:34}
    \begin{split}
      \bm{Z}_{h}^{T} \ddti{M(s)}{s} \bm{\phi}_{h}
      \leq
      % \Lnorm[\surface_{h}(t)][\infty]{\nbgh\cdot V_{h}}
      c \sqrt{\bm{Z}_{h}^{T} M(s)
        \bm{Z}_{h}} \sqrt{\bm{\phi}_{h}^{T} M(s) \bm{\phi}_{h}}, \\
      \bm{Z}_{h}^{T} \ddti{A(s)}{s} \bm{\phi}_{h}
      \leq
      % \Wnorm[\surface_{h}(t)][1,\infty]{V_{h}}
      c\sqrt{\bm{Z}_{h}^{T} A(s) \bm{Z}_{h}}
      \sqrt{\bm{\phi}_{h}^{T} A(s) \bm{\phi}_{h}}.
    \end{split}
  \end{equation}
  In the following with drop the $s$ dependency.  Let $\bm{u}$ be the time
  dependent coefficient vector of $u_{h}$.  Then we have
  \begin{equation*}
    % \label{eq:23}
    0 = - M \ddti{\bm{u}}{s}  + A \bm{u} = - M \ddti{\bm{u}}{s}  + (A+M) \bm{u} -
    M\bm{u}.
  \end{equation*}
  Equivalently we write this equation as
  \begin{multline*}
    -\frac{1}{2} \ddti{}{s}\bigl[ \bm{u}^{T} M(A+M)^{-1}M \bm{u}\bigr] \\
    =  - \bm{u}^{T} M \bm{u} + \bm{u}^{T}M (A +M)^{-1} M \bm{u}
    - \frac{1}{2} \bm{u}^{T} \ddti{}{s}\bigl[M (A+M)^{-1} M\bigr] \bm{u}.
  \end{multline*}
  The last term expanded reads
  \begin{multline*}
    \frac{1}{2}\bm{u}^{T} \ddti{}{s}\bigl[M (A+M)^{-1} M\bigr] \bm{u} \\
    = \bm{u}^{T} \ddti{M}{s} (A+M)^{-1} M \bm{u}
    + \frac{1}{2} \bm{u}^{T} M \ddti{(A+M)^{-1}}{s} M \bm{u}
    = I_{1}+ I_{2}.
  \end{multline*}
  Using \eqref{eq:34} and a Young inequality we
  estimate as
  \begin{align*}
    \abs{I_{1}}
    & \leq c \cdot % \Wnorm[\surface_{h}(t)][1,\infty]{V_{h}}^{2}
      \bm{u}^{T} M (A+M)^{-1} M \bm{u}
      + \frac{1}{2} \bm{u}^{T} M \bm{u}.  \\
    \abs{I_{2}}
    & = \frac{1}{2}\biggl\lvert
      \bm{u}^{T} M (A+M)^{-1} \ddti{(A+M)}{s} (A+M)^{-1} M
      \bm{u} \biggr\rvert \\
    & \leq c \cdot % \Wnorm[\surface_{h}(t)][1,\infty]{V_{h}}
      \bm{u}^{T} M(A+M)^{-1} M \bm{u}
  \end{align*}
  Putting everything together we reach at
  \begin{equation*}
    -\ddti{}{s}\bigl[ \bm{u}^{T} M(A+M)^{-1} M \bm{u} \bigr]
    \leq  - \bm{u}^{T} M \bm{u} + c \cdot \bm{u}^{T}M(A+M)^{-1} M \bm{u}.
  \end{equation*}
  The claim then follows from Lemma~\ref{lemma:modfied_gronwall}.
\end{proof}

\subsection{A weak discrete maximum principle}

%\rev{The authors' reference [25] (Schatz, Thom\'{e}e and Wahlbin) was updated in [Schatz, Thom\'{e}e \& Wahlbin (1998)]. There the Euclidean version of the weak discrete maximum principle (Proposition 5.2) is shown to be logarithm-free.}{
\begin{proposition}
  \label{proposition:weak_max}
  Let $U_{h}(x,t)\in S_{h}(t)$ the ESFEM solution of our linear heat problem.
  Then there exists a constant $c=c(T,v)>0$, which depends exponentially on $T$
  and $v$ such that
  \[
  \Lnorm[\surface_{h}(t)][\infty]{U_{h}(t)} \leq c \abs{\log h}
  \Lnorm[\surface_{h}(0)][\infty]{U_{h}(0)}.
  \]
\end{proposition}
%}

\begin{proof}
  There exists \(x_{h}\in \surface_{h}(t)\) such that
  \begin{multline*}
    \Lnorm[][\infty]{U_{h}(t)}
    = \abs{U_{h}(x_{h},t)}
    = m_{h}\bigl(t; U_{h}(t), \delta_{h}^{t,x_{h}}\bigr)
    = m_{h}\bigl(t; E(t,0)U_{h}^{0}, \delta_{h}^{t,x_{h}}\bigr) \\
    = m_{h}(0; U_{h}^{0}, E(t,0)^{*} \delta_{h}^{t,x_{h}})
    \leq \Lnorm[][\infty]{U_{h}^{0}}
    \Lnorm[][1]{E(t,0)^{*} \delta_{h}^{t,x_{h}}}.
  \end{multline*}
  The claim follows from Lemma~\ref{lemma:technical_weak_max}.
\end{proof}

\begin{lemma}
  \label{lemma:technical_weak_max}
  For $G^{x}_{h}(t,s):= E_{h}(t,s)^{*}\delta^{t,x}_{h}$, where $\delta^{t,x}_{h}$ is
  defined via \eqref{eq:15} and $E_{h}(t,s)^{*}$ is defined via \eqref{eq:21},
  it holds
  \begin{equation*}
    \Lnorm[\surface_{h}(0)][1]{G^{x}(t,0)} \leq c \abs{\log h},
  \end{equation*}
  where the constant $c=c(T,v)$ depending exponentially on $T$ and $v$ such
  and is independent of \(x\), \(h\), \(t\) and \(s\).
\end{lemma}

\begin{proof}
  The proof presented here is a modification of the proof from Schatz,
  Thom\'{e}e and Wahlbin \cite[Lemma~2.1]{1980_stw}.  We estimate
  \begin{equation*}
    \Lnorm[\surface_{h}(0)][1]{G_{h}^{x}(t,0)} \leq
    \Lnorm[\surface_{h}(0)]{1/\pwt^{x}}
    \Lnorm[\surface_{h}(0)]{\pwt^{x} G_{h}^{x}(t,0)}.
  \end{equation*}
  With subsection~\ref{sec:integr-with-polar} it follows
  \begin{equation*}
    \Lnorm[\surface_{h}(0)]{1/\pwt^{x}}^{2} \leq c \abs{\log h}.
  \end{equation*}
  It remains to show
  \begin{equation*}
    \Lnorm[\surface_{h}(0)]{\pwt^{x} G_{h}^{x}(t,0)}^{2} \leq c \abs{\log h}.
  \end{equation*}
  In the following we abbreviate $\pwt= \pwt^{x}$ and $G_{h}=G_{h}^{x}(t,s)$ With
  equation~\eqref{eq:22} and the discrete transport property we proceed as
  follows
  \begin{align*}
    - \frac{1}{2} \ddti{}{s}
    & \Lnorm[\surface_{h}(s)]{\pwt G_{h}}^{2} + \Lnorm[\surface_{h}(s)]{\pwt \nbgh
      G_{h}}^{2}\\
    = {}
    & - m_{h}\bigl(s; \partial^{\bullet,s}_{h} G_{h}, \pwt^{2}G_{h}\bigr) +
    a_{h}\bigl(s; G_{h}, \pwt^{2}G_{h}\bigr) \\
    &- 2m_{h}\bigl(s; \pwt\nbgh
      G_{h}, G_{h} \nbgh \pwt\bigr)   \\
    &- m_{h}\bigl(s;\partial^{\bullet,s}_{h} \pwt, \pwt G_{h}^{2}\bigr)
      - \frac{1}{2}m_{h}\bigl(s;\pwt^{2}G_{h}^{2}, \nbgh\cdot V_{h}\bigr) \\
    = {}
    & - m_{h}\bigl(s; \partial^{\bullet,s}_{h} G_{h}, \pwt^{2}G_{h}-\psi_{h}\bigr) +
      a_{h}\bigl(s; G_{h}, \pwt^{2}G_{h}-\psi_{h}\bigr) \\
    &- 2m_{h}\bigl(s; \pwt\nbgh
      G_{h}, G_{h} \nbgh \pwt\bigr)   \\
    &- m_{h}\bigl(s;G_{h}\partial^{\bullet,s}_{h} \pwt, \pwt G_{h}\bigr)
      - \frac{1}{2}m_{h}\bigl(s;\pwt^{2}G_{h}^{2}, \nbgh\cdot V_{h}\bigr) \\
    = {}
    & I_{1}+ I_{2}+I_{3}+I_{4}+I_{5}.
  \end{align*}
  For the choice $\psi_{h} = \Lproj (\pwt^{2}G_{h}\bigr)$ we have $I_{1}=0$.
  Using Cauchy--Schwarz inequality, Lemma~\ref{lemma:weighted_l_two_proj_err} and an
  inverse estimate~\ref{lemma:inverse_estimate} we get
  \begin{align*}
    \abs{I_{2}} \leq {}
    % ------------------------------------------------------------
    % Details
    % & \Hnorm{G_{h}} \Hnorm{\pwt G_{h} - \Lproj(\pwt G_{h})}\\
    % \leq {}
    % & ch \Hnorm{G_{h}} ( \Lnorm{G_{h}} + \Lnorm{\pwt \nbgh G_{h}}) \\
    % \leq {}
    % & c \Lnorm{G_{h}} ( \Lnorm{G_{h}} + \Lnorm{\pwt \nbgh G_{h}}) \\
    % \leq {}
    % ------------------------------------------------------------
    & c \bigl( \Lnorm[\surface_{h}(s)]{G_{h}}^{2} +
      \Lnorm[\surface_{h}(s)]{G_{h}} \cdot \Lnorm[\surface_{h}(s)]{\pwt
      \nbgh G_{h}}\bigr).
  \end{align*}
  Using Lemma~\ref{lemma:rho_calculation}~\eqref{eq:68} we reach at
  \begin{align*}
    \abs{I_{3}} \leq {}
    % ------------------------------------------------------------
    % Details
    % & \Lnorm{\pwt \nbgh G_{h}} \Lnorm{G_{h} \nbgh \sigma} \\
    % \leq {}
    % ------------------------------------------------------------
    & c \Lnorm[\surface_{h}(s)]{G_{h}} \Lnorm[\surface_{h}(s)]{\pwt
      \nbgh G_{h}}.
  \end{align*}
  Using Lemma~\ref{lemma:rho_calculation}~\eqref{eq:67} we have
  \begin{align*}
    \abs{I_{4}} \leq {}
    % & \Lnorm[\surface_{h}(s)][\infty]{V_{h}}
    %   \Lnorm[\surface_{h}(s)]{G_{h}} \Lnorm[\surface_{h}(s)]{\pwt
    %   G_{h}}, \\
    & c \Lnorm[\surface_{h}(s)]{G_{h}}
        \Lnorm[\surface_{h}(s)]{\pwt G_{h}}, \\
    \abs{I_{5}} \leq {}
    & % \Lnorm[\surface_{h}(s)][\infty]{\nbgh \cdot V_{h}}
      c \Lnorm[\surface_{h}(s)]{\pwt G_{h}}^{2}.
  \end{align*}
  After a Young inequality we have
  \begin{multline*}
    - \ddti{}{s} \Lnorm[\surface_{h}(s)]{\pwt G_{h}}^{2}
    + \Lnorm[\surface_{h}(s)]{\pwt \nbgh G_{h}}^{2} %\\
    \leq c \Lnorm[\surface_{h}(s)]{G_{h}}^{2}
    +c %\Wnorm[\surface_{h}(s)][1,\infty]{V_{h}}
    \Lnorm[\surface_{h}(s)]{\pwt G_{h}}^{2}.
  \end{multline*}
  Lemma~\ref{lemma:modfied_gronwall} yields
  \begin{equation*}
    \Lnorm[\surface_{h}(0)]{\pwt G_{h}(t,0)}^{2}  \leq c
    \biggl( \int_{0}^{t}
    \Lnorm[\surface_{h}(s)]{G_{h}(t,s)}^{2} \mathrm{d}s  +
    \Lnorm[\surface_{h}(0)]{\pwt^{x} \delta^{x}_{h}} \biggl).
  \end{equation*}
  For the first term we get from Lemma~\ref{lemma:adjoint_problem_estimate} and
  Lemma~\ref{lemma:discrete_greens_function} the bound
  \begin{equation*}
    \int_{0}^{t} \Lnorm[\surface_{h}(s)]{G_{h}(t,s)}^{2} \mathrm{d}s
    \leq c \abs{\log h}.
  \end{equation*}
  The last term is bounded according to Lemma~\ref{lemma:delta_rho_bound}.
  % The last term is bounded according to Lemma~\ref{lemma:delta_rho_bound},
  % therefore it remains to show
  % Denote by $\bm{1}$ the constant finite element with value equals $1$.  Then we
  % have w.r. to the $L^{2}$ scalar product an orthogonal decomposition
  % \[
  % G_{h}(t,s) = G_{h}(t,s)^{\perp} + \biggl(\int_{\surface_{h}(s)} G_{h}(t,s)\biggr) \bm{1}.
  % \]
  % Hence we have
  % \begin{align*}
  %   \int_{0}^{t} &\Lnorm[\surface_{h}(s)]{G_{h}(t,s)}^{2} \mathrm{d}s \\
  %   &= \int_{0}^{t}\Lnorm[\surface_{h}(s)]{G_{h}(t,s)^{\perp}}^{2} \mathrm{d}s +
  %   \int_{0}^{t} \vol\bigl(\surface_{h}(s)\bigr) \biggl(\int_{\surface_{h}(s)}
  %   G_{h}(t,s)\biggr)^{2} \mathrm{d}s
  % \end{align*}
  % It is easy to see
  % \begin{align*}
  % \int_{0}^{t} &\vol\bigl(\surface_{h}(s)\bigr) \biggl(\int_{\surface_{h}(s)}
  %                G_{h}(t,s)\biggr)^{2} \mathrm{d}s \\
  %   &\leq \frac{1}{2}\max_{s\in
  %     [0,t]}\vol\bigl(\surface_{h}(s)\bigr) \vol\bigl(\G_{h}Th \bigr) + \frac{1}{2}
  %   \int_{0}^{t} \Lnorm[\surface_{h}(s)]{G_{h}(t,s)}^{2} \mathrm{d}s.
  % \end{align*}
  % Thus we have
  % \[
  % \int_{0}^{t}\Lnorm[\surface_{h}(s)]{G_{h}(t,s)}^{2} \mathrm{d}s \leq c + 2
  % \int_{0}^{t}\Lnorm[\surface_{h}(s)]{G_{h}(t,s)^{\perp}}^{2} \mathrm{d}s.
  % \]
  % This is a consequence of Lemma~\ref{lemma:adjoint_problem_estimate} and
  % Lemma~\ref{lemma:discrete_greens_function}.
\end{proof}

\begin{remark}
By using the techniques of \cite{1998_stw} instead of \cite{1980_stw} the logarithmic factor $|\log(h)|$ is expected to disappear, however this would lead to a much more technical and quite lengthy proof, as already noted in the introduction.
\end{remark}

\section{Convergence of the semidiscretization}
% \section{Maximum norm parabolic error estimates}
\label{sec:conv-semid}

\begin{theorem}
%\label{theorem:convergence_semidiscrete}
  Let $\surface(t)$ be an evolving surface, let $u\colon \surface(t)\to \R$
  be the solution of \eqref{eq_ES-PDE-strong-form} and let
  $u_{h}= U_{h}^{l}\in H^{1}\bigl(\surface(t)\bigr)$ be the solution
  of \eqref{eq: semidiscrete problem}.  If it holds
  \begin{align*}
    \Lnorm[\surface_{h}(t)][\infty]{\Rmap(t)u - U_{h}}
    \leq c h^{2},
  \end{align*}
  then there exists $h_{0}>0$ sufficiently small and $c=c(h_{0})>0$ independent
  of $t$, such that for all $0<h< h_{0}$ we have the estimate
  \begin{align*}
    \Lnorm[\surface(t)][\infty]{u- u_{h}}
    & + h\Wnorm[\surface(t)][1,\infty]{u - u_{h}} \\
    & \leq c h^{2} \abs{\log h}^{4} (1+t)
    \big(\Wnorm[\surface(t)][2,\infty]{u}
    + \Wnorm[\surface(t)][2,\infty]{\mat u}\big).
  \end{align*}
\end{theorem}

\begin{proof}
  It suffices to prove the $L^{\infty}$-estimate, since an inverse inequality
  implies the $W^{1,\infty}$-estimate.  \par
  For this proof we denote by $\Rmap u = \Rmap(t)u$, $\Rmapl u = (\Rmap u)^{l}$
  and $u_{h} = U_{h}^{l}$.  We split the error as follows
  \begin{align*}
    u - u_{h} = (u - \Rmapl u) + (\Rmap u - U_{h})^{l} = \sigma
    + \theta_{h}^{l}.
  \end{align*}
  Because of Theorem~\ref{theorem: L infty errors in Ritz} it remains to bound
  $\theta_{h}$.  Obviously there exists $R_{h} \in S_{h}(t)$ such that for all
  $\phi_{h}\in S_{h}(t)$ it holds
  \begin{equation*}
    \frac{\mathrm{d}}{\mathrm{d}t}
    \int_{\surface_{h}(t)} \theta_{h} \phi_{h}
    + \int_{\surface_{h}(t)} \nbgh \theta_{h} \cdot \nbgh \phi_{h}
    - \int_{\surface_{h}(t)} \theta_{h} \mat_{h} \phi_{h}
    = \int_{\surface_{h}(t)} R_{h} \phi_{h}.
  \end{equation*}
  By the variation of constant formula we deduce
  \begin{align*}
    \theta_{h}(t)
    = E_{h}(t,0) \theta_{h}(0) + \int_{0}^{t} E_{h}(t,s) R_{h}(s) \d s.
  \end{align*}
  With Proposition~\ref{proposition:weak_max} we get
  \begin{equation*}
    \Lnorm[\surface_{h}(t)][\infty]{\theta_{h}}
    \leq c\abs{\log h}\bigl(
      \Lnorm[\surface_{h}(t)]{\theta_{h}(0)}
      + t \max_{s\in [0,t]} \Lnorm[\surface_{h}(t)][\infty]{R_{h}(s)}
      \bigr).
  \end{equation*}
  Observe that if we denote by $\varphi_{h}:= \phi_{h}^{l}$, then a quick
  calculation reveals
  \begin{equation}
    \label{eq:54}
    \begin{aligned}
      m_{h}(R_{h},\phi_{h})
      % ------------------------------------------------------------
      % Details
      % & = \frac{\d }{\d t} m_{h}(\Rmap u, \phi_{h})
      % + a_{h}(\Rmap u, \phi_{h}) - m_{h}(\Rmap u, \mat_{h} u) \\
      % & \quad - \left( \frac{\d }{\d t} m(u, \varphi_{h})
      %   + a(u,\varphi_{h}) - m(u , \mat_{h} u)
      % \right) \\
      % ------------------------------------------------------------
      & = m_{h}(\mat_{h} \Rmap u , \phi_{h})
      + g_{h}(V_{h}; \Rmap u, \phi_{h})
      + a_{h}(\Rmap u, \phi_{h}) \\
      & \quad - \bigl( m ( \mat_{h} u ,\varphi_{h})
      + g(v_{h}; u, \varphi_{h})
      + a(u , \varphi_{h}) \bigr)
    \end{aligned}
  \end{equation}
  Lemma~\ref{lemma:linfty_stability_theta_residual} finishes the proof.
\end{proof}

\begin{lemma}
  \label{lemma:linfty_stability_theta_residual}
  Assume that $R_{h}\in S_{h}(t)$ satisfies for all $\phi_{h}\in S_{h}(t)$ with
  $\varphi_{h} := \phi_{h}^{l}$ equation \eqref{eq:54}.  Then it holds
  \begin{align*}
    \Lnorm[\surface_{h}(t)][\infty]{R_{h}}
    \leq c h^{2} \abs{\log h}^{3}
    (\Wnorm[\surface(t)][2,\infty]{u}
    + \Wnorm[\surface(t)][2,\infty]{\mat u}).
  \end{align*}
\end{lemma}

\begin{proof}
  Using Definition~\ref{definition:L2_proj}~\eqref{eq:47}, \eqref{eq:54}
  and since $L^{\infty}$ is the dual of $L^{1}$ we deduce
  \begin{equation*}
    \Lnorm[\surface_{h}(t)][\infty]{R_{h}}
    = \sup_{\substack{f_{h}\in L^{1}(\surface_{h}(t))\\
        \Lnorm[\surface_{h}(t)][1]{f_{h}}=1}}
    m_{h}(R_{h}, f_{h}) =
    \sup_{\substack{f_{h}\in L^{1}(\surface_{h}(t))\\
        \Lnorm[\surface_{h}(t)][1]{f_{h}}=1}}
    m_{h}(R_{h}, \Lproj f_{h}).
  \end{equation*}
  Now consider
  \begin{align*}
    m_{h}(R_{h},\Lproj f_{h})
    & = m_{h}(\mat_{h} \Rmap u, \Lproj f_{h}) - m(\mat_{h} u, \Lproj f_{h}^{l}) \\
    & \quad + g_{h}(V_{h}; \Rmap u, \Lproj f_{h})
      - g(v_{h}; u, \Lproj f_{h}^{l}) \\
    & \quad + a_{h}(\Rmap u , \Lproj f_{h})
      - a(u , \Lproj f_{h}^{l}) \\
    & = I_{1} + I_{2} + I_{3}.
  \end{align*}
  Using Lemma~\ref{lemma: geometric errors of forms - Holder} and
  Theorem~\ref{theorem:old_L2_bound} it is easy to see
  \begin{align*}
    \abs{I_{1}}
    % ------------------------------------------------------------
    % Details
    % & \leq \abs{m_{h}(\mat_{h} \Rmap u - \mat_{h} u^{-l}, \Lproj f_{h})} \\
    % & \quad + \bigabs{m_{h}(\mat_{h} u^{-l}, \Lproj f_{h})
    %   - m(\mat_{h} u, \Lproj f_{h}^{l})} \\
    % ------------------------------------------------------------
    & \leq c \bigl(
    \Lnorm[\surface(t)][\infty]{\mat_{h} u - \mat_{h} (\Rmap u)^{l}} \\
    & \quad + h^{2}( \Lnorm[\surface(t)][\infty]{\mat u}
    + h^{2}\Wnorm[\surface(t)][1,\infty]{u})
    \bigr) \Lnorm[\surface_{h}(t)][1]{f_{h}} \\
    \abs{I_{2}}
    % ------------------------------------------------------------
    % Details
    % & \leq \abs{g_{h}(V_{h}; \Rmap u - u^{-l}, \Lproj f_{h})} \\
    % & \quad + \abs{ g_{h}(V_{h}; u^{-l},\Lproj f_{h})
    %   - g(v_{h}; u, \Lproj f_{h}^{l})} \\
    % ------------------------------------------------------------
    & \leq c (
    \Lnorm[\surface(t)][\infty]{u - \Rmap u^{l}}
    + h^{2} \Lnorm[\surface(t)][\infty]{u}
    ) \Lnorm[\surface_{h}(t)][1]{f_{h}} \\
    \abs{I_{3}}
    % ------------------------------------------------------------
    % Details
    % & = \lvert \aast_{h}(\Rmap u, \Lproj f_{h})
    %   - \aast(u,\Lproj f_{h}^{l}) \\
    % & \quad - \bigl(m_{h}(\Rmap u, \Lproj f_{h})
    %   - m(u,\Lproj f_{h}^{l})\bigr)\rvert \\
    % & = \abs{m(u, \Lproj f_{h}^{l}) - m_{h}(\Rmap u, \Lproj f_{h})} \\
    % & \leq \abs{m(u,\varphi) - m_{h}(u^{-l},\Lproj f_{h})} \\
    % & \quad + \abs{m_{h}(u^{-l} - \Rmap u , \Lproj f_{h})} \\
    % ------------------------------------------------------------
    & \leq c (
    h^{2}\Lnorm[\surface(t)][\infty]{u}
    + \Lnorm[\surface(t)][\infty]{u - (\Rmap u)^{l} })
    \Lnorm[\surface_{h}(t)][1]{f_{h}}
  \end{align*}
  Theorem~\ref{theorem: L infty errors in Ritz} and Theorem~\ref{theorem: L
    infty errors in mat derivative of Ritz}  imply the claim.
\end{proof}

\section{A numerical experiment}
\label{sec:numerical-experiment}

We present a numerical experiment for an evolving surface parabolic problem
discretized in space by the evolving surface finite element method.  As a time
discretization method we choose backward difference formula 4 with a
sufficiently small time step (in all the experiments we choose
\(\tau = 0.001\)).  \par
As initial surface \(\surface_{0}\)
we choose the unit sphere \(S^{2}\subset \R^{3}\).
The dynamical system is given by
\(\Phi(x,y,z,t) = (\sqrt{1+0.25 \sin(2\pi t)}  x, y, z) \),
which implies the velocity
\(v(x,y,z,t) = (\pi  \cos(2\pi t) / (4 + \sin(2\pi t)) x,0,0 )\), over the time
interval \([0,1]\).
As the exact solution we choose \(u(x,y,z,t) = x y e^{-6t}\).
The complicated right-hand side was calculated using the computer algebra system
Sage \cite{sage}. \par
We give the errors in the following norm and seminorm
\begin{align*}
  L^\infty(L^\infty):
  &\qquad \max_{1\leq n \leq N}\|u_h^n - u(.,t_n)\|_{L^\infty(\Ga(t_n))},\\
  L^2(W^{1,\infty}):
  &\qquad  \Big(\tau \sum_{n=1}^{N}
  \lvert \nb_{\Ga(t_n)}\big(u_h^n - u(.,t_n)\big)\rvert_{L^\infty(\Ga(t_n))}^2
  \Big)^{1/2}.
\end{align*}
The experimental order of convergence (EOC) is given as
\begin{equation*}
  EOC_{k}=\frac{\ln(e_{k}/e_{k-1})}{\ln(2)}, \qquad (k=2,3, \dotsc,n),
\end{equation*}
where \(e_{k}\) denotes the error of the \(k\)-th level.
\begin{table}[!ht]
  \centering
  \begin{tabular}{r r  l l l l}
    \toprule
    level & dof & \(L^{\infty}(L^{\infty})\) & EOCs & \(L^{2}(W^{1,\infty})\) & EOCs \\
    \midrule
    1 & 126   & 0.00918195 & -    & 0.01921707 & - \\
    2 & 516   & 0.00308305 & 1.57 & 0.01481673 & 0.37 \\
    3 & 2070  & 0.00100752 & 1.61 & 0.00851267 & 0.80 \\
    4 & 8208  & 0.00025326 & 1.99 & 0.00399371 & 1.09 \\
%    5 & 32682 & 0.00007752 & 1.71 & 0.00223474 & 0.84 \\
    \bottomrule
  \end{tabular}
   % original data
   % errors
   % 9.181951e-03 1.921707e-02
   % 3.083053e-03 1.481673e-02
   % 1.007524e-03 8.512668e-03
   % 2.532567e-04 3.993712e-03
   % 7.751871e-05 2.234741e-03
   % eoc
   % 1.57444  0.375161
   % 1.61355  0.799544
   % 1.99214  1.09188
   % 1.70798  0.837623
  \caption{Errors and EOCs in the \(L^{\infty}(L^{\infty})\) and
    \(L^{2}(W^{1,\infty})\) norms}
%  \label{table: EOCs}
\end{table}

\section*{Acknowledgement}
The authors would like to thank Prof.\ Christian Lubich for the invaluable
discussions on the topic, and for his encouragement and help during the
preparation of this paper.  We would also like to thank Buyang Li from Nanjing
University  for his discussions on the topic.
The research stay of B.K.~at the University of T\"{u}bingen has been funded by
the Deutscher Akademischer Austausch Dienst (DAAD).

%    Bibliographies can be prepared with BibTeX using amsplain,
%    amsalpha, or (for "historical" overviews) natbib style.
\bibliographystyle{amsplain}
%    Insert the bibliography data here.

\appendix

\section{A Green's function for evolving surfaces}
% \label{sec:greens-funct-evolv}

Aubin \cite[Section~4.2]{Aubinbook} proves existence of a Green's function on a
closed manifold \(M\), that is a function which satisfies in \(M\times M\)
\begin{equation*}
  \laplace_{Q\text{ distr.}} \green(P,Q) = \delta_{P}(Q),
\end{equation*}
where \(\laplace\) is the Laplace--Beltrami operator on \(M\).  The Green's
function is unique up to an constant.  For
Lemma~\ref{lemma:improved_inverse_estimate} we need that the first derivative of
a Green's function can be bounded independent of \(t\).

\begin{theorem}[Green's function]
  \label{thm:green}
  Let $\surface(t)$ with $t\in [0,T]$ be an evolving surface.  There exists
  a Green's function \(\green(t;x,y)\) for \(\surface(t)\).  The value of
  $\green(x,y)$ depends only on the value of $\riemdist(x,y)$.
  \(\green(x,y)\)satisfies the   inequality
  \begin{equation*}
    \abs{\nbg^{x} \green(t;x,y) } \leq c \frac{1}{\riemdist(x,y)}.
  \end{equation*}
  for some $c>0$ independent of $t$.  \par
  Furthermore for all functions $\varphi\in C^{2}(\GT)$ it holds
  \begin{equation}
    \label{eq:8}
    \varphi(x,t) = \frac{1}{V} \int_{\surface(t)} \varphi(y,t)\mathrm{d}y -
    \int_{\surface(t)} \green(t;x,y) \lb\varphi(y,t) \mathrm{d}y.
  \end{equation}
\end{theorem}

\begin{proof}
  As noted in Aubin \cite[4.10]{Aubinbook} the distance \(r = \riemdist(x,y)\)
  is only a Lipschitzian function on \(\surface(t)\).  To use his construction
  we therefore need to revise that the injectivity radius at any point \(P\in
  \surface(t)\) can be bounded by below by a number independent of \(P\) and
  \(t\).  This follows if the Riemannian exponential map is continuous in \(t\)
  and from Lemma~\ref{lemma:inv_fct_thm_parameter}.  To prove that the
  Riemannian exponential map is continuous one carefully revises the
  construction of exponential map as it is given in Chavel
  \cite[Chapter~1]{2006_ch}.  Formula~\eqref{eq:8} follows from Aubin
  \cite[Theorem~4.13]{Aubinbook} and that the constant is independent of \(t\)
  is a straightforward calculation.
\end{proof}

\section{Calculations with some weight functions on evolving surfaces}
\label{sec:extens-some-auxil}

\subsection{Integration with geodesic polar coordinates on evolving surfaces}
\label{sec:integr-with-polar}

Assume we have sufficiently smooth function \(f\colon \surface(t) \times
\surface(t) \to \R\), where the value \(f(x,y)\) depends on the distance \(r =
\riemdist(x,y)\) and we want to estimate the quantity \(\int_{\surface(t)}
f(x,y) \d y \) for a fix \(y\).  \par
Applying the well known coarea formulae to the distance function \(r\),
cf.\ Chavel \cite[Theorem~3.13]{2006_ch} and Morgan
\cite[Theorem~3.13]{morgan}, we reach at
\begin{multline*}
  \int_{\surface(t)} f(x,y) \d y = \int_{0}^{\infty}
  \int_{\{\riemdist(x,y)=r\}} f(r) \d \omega \d r \\
  = \int_{0}^{\infty}
  \frac{\mathcal{H}^{m}\bigl(\{\riemdist(x,y)=r\}\bigr)}{r^{m}}  f(r) r^{m} \d r,
\end{multline*}
where $\mathcal{H}^{m}$ denotes the $m$-dimensional Hausdorff measure.  If
\(\surface(t)\) is not a closed surface but \(\R^{m+1}\) then
\(\frac{\mathcal{H}^{m}\bigl(\{\riemdist(x,y)=r\}\bigr)}{r^{m}}\) would be
constant.  For closed surfaces the situation is different.
Obviously there exists a positive number \(R>0\) independent of \(t\) and
\(x,y\in \surface(t)\) such that for all \(r \geq R\) it holds
\begin{equation*}
  \mathcal{H}^{m}\bigl(\riemdist(x,y)=r\bigr) = 0.
\end{equation*}

\begin{lemma}
  \label{lemma:geodesic_sphere}
  There exists $c>0$ depending on $t\in [0,T]$ and $x\in \surface(t)$ such that
  \begin{equation*}
    \sup_{r>0} \frac{\mathcal{H}^{m}\bigl(\{\riemdist(x,y)=r\}\bigr)}{r^{m}}
    \leq c.
  \end{equation*}
\end{lemma}

\begin{proof}
  It holds
  \begin{equation*}
    \lim_{r\to 0} \frac{\mathcal{H}^{m}\bigl(\{\riemdist(x,y)=r\}\bigr)}{r^{m}} =
    \omega_{m},
  \end{equation*}
  where $\omega_{m}$ is the volume of the $m$-dimensional sphere in $\R^{m+1}$,
  cf.\ Gray \cite[Theorem~3.1.]{geodesic_ball}.  Thus the proof is
  finished if find a $c>0$ such that
  \begin{equation*}
    \mathcal{H}^{m}\bigl(\{\riemdist(x,y)=r\}\bigr) \leq c \quad \forall r \in
    [0,\infty).
  \end{equation*}
  This can be seen as follows.  For a fix point \(p\in \surface(t)\) it is
  possible to use the Riemannian exponential to flat out \(\surface(t)\), cf.\
  Figure~\ref{fig:wp_set_explanation} for an illustration on the torus.
  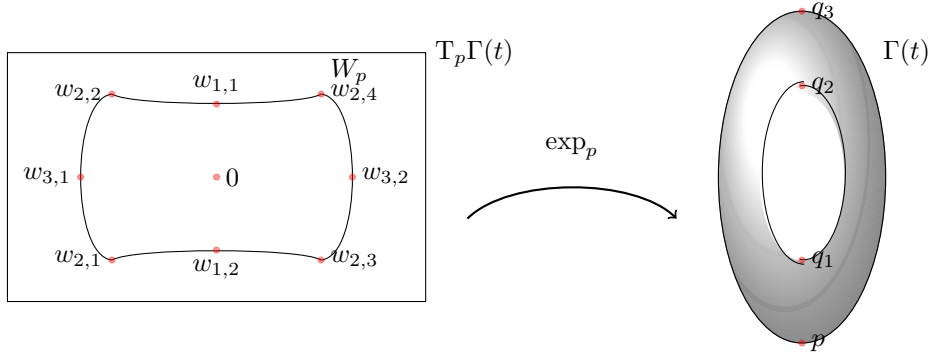
\begin{figure}
    \centering
    \begin{tikzpicture}[scale=.55]
      % RHS
      \begin{scope}[xshift = 3cm]
        % Torus H\"{u}lle
        \path[ball color = white, draw = black, opacity=0.4] (1,1) circle [x radius
        = 2, y radius = 4];
        \draw (1,1) circle [x radius = 2, y radius = 4];
        \path (3.5,4) node {$\surface(t)$};
        % Torus Loch
        \filldraw[white] (1,1) circle [x radius =1, y radius = 2];
        % 3D -Effekt
        % \draw (1.05,-.1) arc [x radius =.5, y radius = 1.1, start angle = 270, end
        % angle = 90];
        % \draw (.95,2) arc [x radius =.5, y radius = 1, start angle = 95, delta angle
        % = -190];
        \draw (1.05,-1.1) arc [x radius =1, y radius = 2.2, start angle = 270, end
        angle = 90];
        \draw (.95,3.2) arc [x radius =1, y radius = 2.1, start angle = 95, delta angle
        = -190];

        % Punkt auf dem Torus
        % \filldraw (1.6,2.2) circle (2pt) node [right = .5cm] {$f(x,y,z)=z$};
        % Alle kritische Punkte
        \foreach \y/\l in { {-3/p}, {-1/q_{1}}, {3.2/q_{2}}, {5/q_{3}}}
        {
          \filldraw[red,opacity=0.4] (1,\y) circle (2pt);
          \path (1,\y) node[anchor = west] {$\l$};
        }
      \end{scope}

      % exp arrow
      \draw[->, thick] (-4,0) .. controls (-3,1) and (0,1) ..  (1,0);
      \path (-1.5,1.2) node [anchor = south] {$\exp_{p}$};

      % LHS
      \begin{scope}[xshift = -10cm, yshift = 1cm]
        % tangent space
        \draw (-5,-3) rectangle (5,3) node [anchor = west]
        {$\mathrm{T}_{p}\surface(t)$};
        % W_p set
        %- left
        \draw (-2.5,-2) .. controls (-3.5,-1.9) and (-3.5,1.9) .. (-2.5,2);
        %- right
        \draw (2.5,-2) .. controls (3.5,-1.9) and (3.5,1.9) .. (2.5,2);
        %- bottom
        \draw (-2.5,-2) .. controls (-2,-1.7) and (2,-1.7) ..  (2.5,-2);
        %- top
        \draw (-2.5,2) .. controls (-2,1.7) and (2,1.7) ..  (2.5,2)
        node [anchor = south west] {$W_{p}$};
        % critical points
        %- p
        \filldraw[red,opacity=0.4] (0,0) circle (2pt);
        \path (0,0) node[anchor = west] {$0$};
        %- w_1
        \filldraw[red,opacity=0.4] (0,1.76) circle (2pt);
        \filldraw[red,opacity=0.4] (0,-1.76) circle (2pt);
        \path (0,1.76) node[anchor = south] {$w_{1,1}$};
        \path (0,-1.76) node[anchor = north] {$w_{1,2}$};
        %- w_2
        \filldraw[red,opacity=0.4] (-2.5,-2) circle (2pt);
        \filldraw[red,opacity=0.4] (-2.5,2) circle (2pt);
        \filldraw[red,opacity=0.4] (2.5,2) circle (2pt);
        \filldraw[red,opacity=0.4] (2.5,-2) circle (2pt);
        \path (-2.5,-2) node[anchor = east] {$w_{2,1}$};
        \path (-2.5,2) node[anchor = east] {$w_{2,2}$};
        \path (2.5,-2) node[anchor = west] {$w_{2,3}$};
        \path (2.5,2) node[anchor = west] {$w_{2,4}$};
        %- w_3
        \filldraw[red,opacity=0.4] (-3.25,0) circle (2pt);
        \filldraw[red,opacity=0.4] (3.25,0) circle (2pt);
        \path (-3.25,0) node[anchor = east] {$w_{3,1}$};
        \path (3.25,0) node[anchor = west] {$w_{3,2}$};
      \end{scope}
    \end{tikzpicture}
    \caption{Illustration of a possible $W_{p}$ for the Torus as a subset of
      $\R^{3}$ with induced metric.  Note that the opposite boundary of $W_{p}$
      are identified.  It holds $\exp_{p}(w_{i,*})= q_{i}$ and $\exp_{p}(0)=p$.
    }
    \label{fig:wp_set_explanation}
  \end{figure}
  We make this argument precise.  \par
  For $r\in [0,\infty)$ let
  \begin{equation*}
    S_{p}(r) := \{ v\in \mathrm{T}_{p} \surface(t) \mid g_{p}(v,v) = r^{2} \}
  \end{equation*}
  be the sphere of radius $r$ and for $v\in S_{p}(1)$ consider the geodesic
  \begin{equation*}
    f_{v}\colon [0,\infty ) \to  \surface, \quad \lambda \mapsto \exp_{p}(\lambda
    v).
  \end{equation*}
  It is well known that a geodesic is just locally length minimizing.  Hence
  there exists a unique $\lambda_{*}(v)>0$, such that
  $f_{v}\rvert [0,\lambda_{*}(v)]$ is a length minimizing geodesic and for every
  $\varepsilon >0$ $f_{v}\rvert [0,\lambda_{*}(v) + \varepsilon]$ is not anymore
  length minimizing.  We define
  \begin{equation*}
    W_{p}(t) := \{ w\in \mathrm{T}_{p}\surface(t)\mid w = \lambda \cdot v \text{ with
    }v\in S_{p}\text{ and } \lambda \in [0, \lambda_{*}(v)]\}.
  \end{equation*}
  Obviously it holds for every \( w\in W_{p}(t)\)  that \(\riemdist\bigl(p,\exp_{p}(w)\bigr) \leq
  R\).  Further there exists for every $q\in \surface$ a unique $w\in W_{p}$
  with $\exp_{p}(w) = q$.  Clearly it holds
  \begin{equation*}
    \exp_{p}\bigl( W_{p} \cap S_{p}(r) \bigr) = \{\riemdist(x,y)=r\}.
  \end{equation*}
  Now apply a general Area-coarea Formula, cf.\ \cite[Theorem~3.13]{morgan}, to
  finish the proof.
\end{proof}

Using this lemma we have the estimate
\begin{equation*}
  \int_{\surface(t)} f(x,y) \d y  \leq c \int_{0}^{R} r^{m}f(r) \d r.
\end{equation*}

\subsection{Comparison of extrinsic and intrinsic distance}

\begin{lemma}
  \label{lemma:comparision_extr_intr_dist}
  There exists a constant $c>0$ independent of $t$ such that for all $x,y\in
  \surface(t)$ the following inequality holds
  \begin{align}
    \label{eq:3}
    c\cdot  \riemdist(x,y) \leq \abs{x-y}.
  \end{align}
\end{lemma}

\begin{proof}
  For simplicity we assume that \(\surface(t)=\surface_{0}\) for all \(t\in
  [0,T]\).  The basic idea is to find a radius \(r>0\) and two constant
  \(c_{1},c_{2}>0\) such that \eqref{eq:3} holds with \(c_{1}\) for
  \(\riemdist(x,y)\leq r\) and with \(c_{2}\) for \(\riemdist(x,y) \geq
  r\).  \par
  Observe that from the compactness from \(\surface_{0}\) it follows that there
  exists \(r>0\) such that for all \(\riemdist(x,y)\leq r\) it holds
  \begin{equation*}
    \nu(x) \cdot \nu(y) \geq \cos(\pi /6).
  \end{equation*}
  After rotation we may assume $x=0$, $\nu(x) = e_{n+1}$ and that
  $\surface_{0}$ may be written as graph of a smooth function, that means that
  there exits $f\colon U(x)\to \R$ smooth with $U(x)\subset \R^{n}$ an open subset,
  such that $z=(z',w)\in \surface_{0}\subset \R^{m}\times \R$ with
  $\riemdist(z,x)\leq r$ if and only if $z'\in U(x)$ and $w = f(z')$.  For
  $x=(0,0)$ and
  $y=\bigl(y',f(y')\bigr)$ consider the path $t\mapsto \bigl(ty',f(ty')\bigr)$.
  We calculate
  \begin{equation*}
    \riemdist(x,y) \leq \int_{0}^{1} \sqrt{1+ \d f_{ty'} y'} \d t \leq \sqrt{1+
    \lVert f \rVert_{W^{1,\infty}}^{2} } \abs{y'} \leq \sqrt{1+
    \lVert f \rVert_{W^{1,\infty}}^{2} } \abs{y-x}.
  \end{equation*}
  Now the derivatives of $f$ are bounded by $m\cdot \tan(\pi/6)$.  \par
  To get the existence of \(c_{2}>0\) observe that \(\riemdist\) is continuous
  and hence the set \(\riemdist^{-1}\{ r>0\}\) is compact.  On this set the
  function \(\abs{x-y}\) does not vanish and takes it maximum and minimum.
\end{proof}

\subsection{Weight functions}

\begin{definition}
  Let $\mu$ and $\widetilde{\mu}$ be like \eqref{eq:12} resp.\ \eqref{eq:10}.  For given
  $\mu$, $\widetilde{\mu}$ with curve $y=y(t)$, we define a curve
  $y_{h}=y_{h}(t):= y(t)^{-l}\in \surface_{h}(t)$.  Now we define
  a weight function on the discrete surface
  \begin{align*}
    \mu_{h}\colon \surface_{h}(t)\to \R,\quad \text{resp.}\quad
    \widetilde{\mu}_{h} \colon \surface_{h}(t) \to \R,
  \end{align*}
  via the same formula like \eqref{eq:12} resp.\ \eqref{eq:10}.
\end{definition}
\begin{lemma}
  There exists a constant $h_{0}=h_{0}(\gamma)>0$ sufficiently small and
  $c=c(h_{0})>0$ independent of $t$ and $h$ such that for all $0< h < h_{0}$ it
  holds
  \begin{align}
    \label{eq:41}
    \frac{1}{c} \mu \leq \mu_{h}^{l} \leq c \mu, \\
    \label{eq:42}
    \frac{1}{c} \widetilde{\mu}
    \leq \widetilde{\mu}_{h}^{l} \leq c \widetilde{\mu}.
  \end{align}
\end{lemma}
\begin{proof}
  The main idea is to observe that we have the inequalities
  \begin{align*}
    \abs{x^{-l}- y_{h}}
    & \leq 2 d + \abs{x - y}, \\
    \abs{x - y}
    & \leq 2 d + \abs{x^{-l} - y_{h}},
  \end{align*}
  where $d = d(t) := \max_{x\in\surface(t)}
  \dist_{\R^{n+1}}\bigl(x,\surface_{h}(t)\bigr)$.
  % ------------------------------------------------------------
  % Details
  % \begin{align*}
  %   \mu_{h}^{l}(x)
  %   & = \abs{x^{-l} - y_{h}}^{2} + \rho^{2} \\
  %   & \leq (2 d + \abs{x-y})^{2} + \rho^{2} \\
  %   & \leq 4 ( 4 d^{2} + \abs{x-y}^{2}) + \rho^{2} \\
  %   & \leq 16 ( d^{2} + \abs{x-y}^{2} + \rho^{2}) \\
  %   & \overset{h < h_{0}}{\leq} 16( \abs{x-y}^{2}+ 2\rho^{2}) \\
  %   & \leq 32 ( \abs{x-y}^{2} + \rho^{2}) = 32 \mu(x).
  % \end{align*}
  % ------------------------------------------------------------
\end{proof}

\section{Modified analytic results for evolving surface problems}

\begin{lemma}[modified Gronwall inequality]
  \label{lemma:modfied_gronwall}
  Let $c>0$ be a positive constant, let $\varphi,\psi$ and $\rho$ be some
  positive functions defined on $[t,T]$ and assume for all $s\in [t,T]$ we have
  the inequality
  \begin{equation*}
    - \ddti{\varphi}{s}(s)  + \psi(s) \leq c \varphi(s) + \rho(s).
  \end{equation*}
  Then it holds
  \begin{equation*}
    \varphi(t) + \int_{t}^{T} \psi(s)\mathrm{d}s \leq e^{c(T-t)} \biggl(
    \varphi(T) + \int_{t}^{T}\rho(s)\mathrm{d}s\biggr).
  \end{equation*}
\end{lemma}

\begin{proof}
  Calculate $-\ddti{}{s}[ \varphi e^{-c(T-s)}]$
  and integrate from $t$ to $T$.
  % \begin{align*}
  %   -\ddti{}{s}[ \varphi e^{-c(T-s)}] = e^{-c(T-s)}\Bigl( -c\varphi(s) -
  %   \ddti{\varphi}{s}(s)\Bigr) \leq  e^{-c(T-s)} \bigl( \rho(s) -
  %   \psi(s)\bigr),\\
  %   -\varphi(T) + \varphi(t)e^{-c(T-t)} \leq  \int_{t}^{T}\rho(s) -
  %   e^{-c(T-s)}\psi(s) \mathrm{d}s \\
  %   \varphi(t)e^{-c(T-t)}\Bigl( \varphi(t) + \int_{t}^{T}
  %   \psi(s)\mathrm{d}s\Bigr) \leq \varphi(T) + \int_{t}^{T}\rho(s) \mathrm{d}s
  % \end{align*}
\end{proof}

\begin{lemma}[modified inverse function theorem]
  \label{lemma:inv_fct_thm_parameter}
  Let $f\colon \R^{n}\times [0,T]\to \R^{n}$ be a smooth map, denote by
  $f(t)(x) := f(x,t)$ and assume that for all $t\in [0,T]$ the map
  $\mathrm{d}f(t)_{0} = \ddxi{f}{x}(0,t)$ is invertible.  Then there exists
  $r>0$ independent of $t$ such that
  \[
  f(t)\colon f(t)^{-1}\{ B_{r}(0)\} \to \R^{n},\quad x\mapsto  f(x,t),
  \]
  is a diffeomorphism onto its image and we have
  $B_{r/2}(0)\subset f(t)^{-1}\{ B_{r}(0)\}$ for all $t$, where
  $B_{r}(0):=\{ x\in \R^{n}\mid \abs{x}\leq r\}$.  The map
  \[
  g \colon [0,T]\times B_{r}(0) \to \R^{n}, \quad
  (t,x)\mapsto f(t)^{-1}(x)
  \]
  is smooth.  In particular $g$ is smooth in $t$.
\end{lemma}

\begin{proof}
  The results follows from the compactness of $[0,T]$ and the smoothness of $f$.
\end{proof}

\end{document}